\documentclass[11pt]{article}
\usepackage{amsfonts}
\usepackage{amssymb,latexsym}
\usepackage{graphicx,amsmath,amscd}
\usepackage[pdftitle={Paper}, pdfauthor={Olivier Brahic}]{hyperref}
\usepackage{stmaryrd}
\usepackage{hyperref}
\usepackage{enumerate}
\usepackage{amssymb}
\usepackage{amsthm}
\usepackage{xspace}
\usepackage{array}
\usepackage[T1]{fontenc}
\usepackage{enumitem}
\usepackage[all]{xy}

\newcommand{\diagramh}[1]{{#1}}
\newtheorem{lemma} {Lemma} [section]
\newtheorem{proposition} [lemma] {Proposition}

\newtheorem{theorem} [lemma] {Theorem}

\newtheorem{definition}[lemma] {Definition}

\theoremstyle{definition}
\newtheorem{remark}[lemma]{Remark}
\newtheorem{theoremz} [] {Theorem}
\newtheorem{example}[lemma] {Example}

\renewenvironment{proof}[1][\unskip]{{\sc Proof #1:\,}}{{\hspace*{\fill} $\square$\\}}

\numberwithin{}{}

\def\d{{\partial}}                  %
\def\dC{{\partial}}                 
\def\g{{\mathfrak{g}}}              
\def\Mon{\mathcal{N}}               
\def\MON{\widetilde{\mathcal{N}}}   
\def\K{\mathcal{K}}                 

\DeclareMathOperator\Der{Der}       
\DeclareMathOperator\im{Im}         
\DeclareMathOperator\Hom{Hom}       

\def\G{\mathcal{G}}                 

\def\s{\mathbf{s}}                    
\def\t{\mathbf{t}}                    
\def\quotient#1#2{%
    \raise0.5ex\hbox{$#1$}\big/\lower0.5ex\hbox{$#2$}%
}

\newcommand{\vba}{$\mathcal{VB}$-algebroid\xspace }                     
\newcommand{\vbas}{$\mathcal{VB}$-algebroids\xspace}                    
\newcommand{\vbg}{$\mathcal{VB}$-groupoid\xspace }                      
\newcommand{\vbgs}{$\mathcal{VB}$-groupoids\xspace }                     
\newcommand{\twoheadlongrightarrow}{\relbar\joinrel\twoheadrightarrow}  
\newcommand{\tto}{\rightrightarrows}                                    
\newcommand{\set}[1]{\left\{#1\right\}}					

\newcommand\HH{\mathcal{H}}
\newcommand\D{\mathcal{D}}
\newcommand\R{\mathbb{R}}

\newcommand\Lie{\mathcal{L}}
\newcommand{\E}{\mathcal{E}}
\newcommand{\Pa}{\mathcal{P}}                             
\newcommand{\bijar}[1][]{%
 \ar[#1]
 \ar@<0.7ex>@{}[#1]|-*=0[@]{\sim}}

\def\Vec{\bf{Vect}}
\def\Term{\bf{Term}}
\def\tTerm{\mathbf{2}\text{\bf{-}}\Term}
\def\tVec{\mathbf{2}\text{\bf{-}}\Vec}

\DeclareMathOperator{\End}{End}
\DeclareMathOperator{\id}{id}
\DeclareMathOperator{\Hor}{Hor}
\DeclareMathOperator{\ad}{ad}
\DeclareMathOperator{\Gau}{Gau}
\DeclareMathOperator{\Iso}{Iso}
\DeclareMathOperator{\hor}{hor}

  \newcommand{\AC}{A^{\scriptscriptstyle{C}}} 
\renewcommand{\AE}{A^{\scriptscriptstyle{E}}} %
  \newcommand{\BC}{B^{\scriptscriptstyle{C}}} 
  \newcommand{\BE}{B^{\scriptscriptstyle{E}}} %

\DeclareMathOperator\hol{{hol}}       

\newcommand{\holE}{\hol^{\scriptscriptstyle{E}}}     
\newcommand{\holC}{\hol^{\scriptscriptstyle{C}}}     
\newcommand{\holK}{\hol^{\scriptscriptstyle{\K}}}     
\newcommand{\omegaE}{{\omega_{\scriptscriptstyle{E}}}} 
\newcommand{\omegaC}{{\omega_{\scriptscriptstyle{C}}}} 
\newcommand{\nablaE}{\nabla^{\scriptscriptstyle{E}}} 
\newcommand{\nablaC}{\nabla^{\scriptscriptstyle{C}}} 
\def\omk{{\omega}_{\scriptscriptstyle{\K}}}          
\newcommand{\DeltaE}{\Delta^{\scriptscriptstyle{E}}} 
\newcommand{\DeltaC}{\Delta^{\scriptscriptstyle{C}}} 

\newcommand\sig{\sigma}

\newcommand{\pE}{p_{\scriptscriptstyle{E}}} 
\newcommand{\pC}{p_{\scriptscriptstyle{C}}} 

\newcommand{\tGau}{2\text{-}\mathrm{Gau}}
\newcommand{\tPa}{2\text{-}\Pa}

\newcommand{\aK}{v}

\newcommand{\sV}{\s_{\scriptscriptstyle{\mathrm{\bf V}}}}
\newcommand{\tV}{\t_{\scriptscriptstyle{\mathrm{\bf V}}}}
\newcommand{\sH}{\s_{\scriptscriptstyle{\mathrm{\bf H}}}}
\newcommand{\tH}{\t_{\scriptscriptstyle{\mathrm{\bf H}}}}
\newcommand{\ts}{\tilde{\s}}
\newcommand{\tit}{\tilde{\t}}
\newcommand{\tsV}{\tilde{\s}_{\scriptscriptstyle{\mathrm{\bf V}}}}
\newcommand{\ttV}{\tilde{\t}_{\scriptscriptstyle{\mathrm{\bf V}}}}
\newcommand{\tsH}{\tilde{\s}_{\scriptscriptstyle{\mathrm{\bf H}}}}
\newcommand{\ttH}{\tilde{\t}_{\scriptscriptstyle{\mathrm{\bf H}}}}

\newcommand{\circH}{\underset{{}^{\mathrm{\bf H}}}{\bullet}}
\newcommand{\circV}{\underset{{}^{\mathrm{\bf V}}}{\bullet}}

\begin{document}

\author{O. Brahic\thanks{{\tt brahicolivier@gmail.com}, {Departamento de Matem\'atica, Universidade Federal do Paran\'a, Setor de Ci\^encias Exatas - Centro Polit\^ecnico 81531-990 Curitiba - Brasil.. This author was supported by CNPq 
grant 401253/2012-0.}}, C. Ortiz\thanks{\texttt{cortiz@ime.usp.br}, {Instituto de Matem\'atica e Estat\'istica, Universidade de S\~ao Paulo, Rua do Mat\~ao 1010, Cidade Universit\'aria, 05508-090, S\~ao Paulo - Brasil.}}}

\date{\today}
\title{Integration of $2$-term representations up to homotopy via $2$-functors}
\maketitle
\begin{abstract}

 Given a representation up to homotopy of a Lie algebroid on a $2$-term complex of vector bundles, we define the corresponding holonomy as a strict $2$-functor from a Weinstein path $2$-groupoid to the gauge $2$-groupoid of the underlying $2$-term complex. We construct a corresponding transformation $2$-groupoid and we prove that the $1$-truncation of this $2$-groupoid is isomorphic to the Weinstein groupoid of the \vba associated to a representation up to homotopy. As applications, we describe alternative integration schemes for semi-direct products of Lie 2-algebras and string algebras.

\end{abstract}

\tableofcontents

\section{Introduction}

Given a Lie algebroid $A\to M$, a representation of $A$ is a vector bundle $E\to M$ together with a Lie algebroid morphism $A\to D(E)$, where $D(E)$ is the Lie algebroid of derivations of $E$. This is equivalent to an $A$-connection $\nabla:\Gamma(A)\times \Gamma(E)\to \Gamma(E); (a,e)\mapsto \nabla_ae$ on the vector bundle $E$, with vanishing curvature, i.e. $\nabla_{[a,b]}-[\nabla_a,\nabla_b]=0$ for every $a,b\in\Gamma(A)$. There are two objects that one can associate to a representation of a Lie algebroid, namely:

\begin{itemize}
\item the cohomology $H^{\bullet}(A,E)$ of $A$ with coefficients in $(E,\nabla)$,
\item the \textbf{transformation Lie algebroid} $ A\ltimes E$, which is a Lie algebroid over $E$.
\end{itemize}

The global counterpart of a representation of a Lie algebroid corresponds to the notion of representation of a Lie groupoid. Given a Lie groupoid $\G$ over $M$, a representation of $\G$ consists on a vector bundle $E\to M$ together with a Lie groupoid morphism $\G\to \Gau(E)$, where $\Gau(E)$ is the gauge groupoid of $E$, i.e. objects are points of $M$ and arrows from $x$ to $y$ are linear isomorphisms $E_x\to E_y$ between the fibers. In this setting, the constructions describe above are well defined as well, so to any Lie groupoid representation one associates two objects:

\begin{itemize}
\item the groupoid cohomology $H^\bullet(\G,E)$ with coefficients in $E$, 
\item the \textbf{transformation groupoid} $ \G\ltimes E$, which is  a Lie groupoid over $E$.
\end{itemize}

Although these constructions contain essentially the same information, one may think of the construction of a  cohomology with coefficients as an algebraic counterpart to the notion of representation, while the transformation algebroid/groupoid is purely geometric.

In this context, one can discuss the integration problem in the following way. Regarding the transformation object, one can pass from Lie groupoids to Lie algebroids by differentiation and, provided $A$ is integrable, one can pass from Lie algebroids to Lie groupoids as well, by an integration process. Moreover, both operations of integration and that of taking the corresponding transformation object commute, which can be summarized in  a simple diagram, as follows:
$$
\xymatrix@C=70pt@R=30pt{
\txt<10pc>{Lie algebroid  \\representation:\\ 
                           $A\circlearrowright  E$\\ \ } \ar[r]^{\txt{integration}}  
                                                    \ar[d]_{\txt{\\transformation\\object\\ }}
                                                                                         & \txt<10pc>{ Lie groupoid \\representation:\\ 
                                                                                                       $\G(A)\circlearrowright  E$ \\ \ }  \ar[d]^{\txt{\\transformation\\object\\}}\\
       \txt<10pc>{ \\ transformation\\ Lie algebroid:\\
                       $A\ltimes  E$ } \ar[r]^{\txt{integration}}  
                                                                                         & \txt<10pc>{\\ transformation\\ Lie groupoid: \\
                                                                                                       $\G(A)\ltimes E=\G(A\ltimes E)$}}$$
                                                                                                      
Note that given a representation of a Lie algebroid/groupoid over $M$, one can also build the corresponding \textbf{semi-direct product} Lie algebroid (groupoid respectively), which is an algebroid (groupoid respectively) over $M$. In this case, the semi-direct product operation commutes with both the Lie functor and integration. In this paper we are concerned only with the transformation object associated to representations.


Unlike Lie algebras, the adjoint representation of a Lie algebroid is not well defined. As shown in \cite{AriasCrainic}, the adjoint representation of a Lie algebroid is well defined as a representation up to homotopy. This is a generalization of the notion of representation of a Lie algebroid, where a Lie algebroid is represented on a graded vector bundle $\E=\oplus E_i$ rather than on a single vector bundle. A representation up to homotopy of a Lie algebroid $A\to M$ on a graded vector bundle $\E$ is defined as a degree one operator on the complex $\Omega^\bullet(A)\otimes \Gamma(\E)$, which squares to zero and satisfies a natural derivation condition. See \cite{AriasCrainic}) for a precise definition.

The notion of representation up to homotopy of a Lie algebroid extends naturally, from the cohomological viewpoint, the usual concept of Lie algebroid representation. In the special case of $2$-term representations up to homotopy of a Lie algebroid $A$, i.e. the graded vector bundle $\E$ is concentrated in degrees -1 and 0, there is a nice geometric interpretation in terms of \vbas, due to Gracia-Saz and Mehta \cite{GM08}. The concept of \vba was introduced by Pradines in \cite{Pradines} and it roughly corresponds to a vector bundle object in the category of Lie algebroids.

In this work, we will only be interested in the case of a representation up to homotopy on a graded vector bundle $\E$ concentrated in degree $-1$ and $0$, $\E=E_{-1}\oplus E_{0}$. Since we want to break free from the cohomological approach, and in order to relate with the work of Gracia-Saz and Mehta \cite{GM08},  we shall forget about degrees, and denote by $E:=E_{-1}$, and by  $C:=E_{0}$ the degree $-1$ and $0$ components, respectively. It is shown in \cite{GM08} that a representation up to homotopy of $A$ on $\E=E\oplus C$ induces a \vba structure on the double vector bundle $D=A\oplus E\oplus C$. Conversely, given a \vba $D$ over $A$, with core $C$ and side $E$, the choice of a splitting $D\cong A\oplus E\oplus C$, induces a $2$-term representation up to homotopy of $A$ on $E\oplus C$. In fact, there is an equivalence of categories between $2$-term representations up to homotopy of a Lie algebroid $A$ and \vbas over $A$ (see \cite{GM08, DJO}). For this reason, we will denote by $A\ltimes \E$ the \vba structure associated to a $2$-term representation up to homotopy of $A$, and think of the operation of passing from a $2$-term representation to homotopy of $A$ to a \vba (with a splitting) as that of building a transformation Lie algebroid.

 This description makes sense for groupoids as well. A representation up to homotopy of a Lie groupoid $\G$ can be defined similarly, as a degree $1$ operator on a complex of the form $C^\bullet(\G,\E)$, which squares to zero. Then in the category of Lie groupoids, the obvious object to play the role of a transformation object is the corresponding \vbg structure, defined on $\t^*C\oplus_\G \s^*E$, which we denote by $\G\ltimes \E$. What plays the role of a splitting for a $2$-term representation up to homotopy of a Lie groupoid is called \emph{right-horizontal lift} in \cite[Sec. 3]{GM10}.

Besides their role as transformation objects, \vbas and \vbgs bring a deep understanding on the representation theory of Lie algebroids and groupoids. In particular, while the adjoint and coadjoint representations up to homotopy of a Lie algebroid (groupoid respectively) depends on the choice of a connection, the \vbas (\vbgs respectively) corresponding to these representations up to homotopy are the tangent and cotangent \vbas (\vbgs respectively), which are entirely canonical. This relation between \vbas and representations up to homotopy has been used to study the infinitesimal picture of Lie groupoids equipped with multiplicative structures, e.g. multiplicative foliations \cite{DJO}.

Back to our discussion, one may expect to be able to produce an integration process, so as to pass from a $2$-term representation up to homotopy of a Lie algebroid to a $2$-term representation up to homotopy of a Lie groupoid, in such a way that this operation commutes with that of taking the corresponding transformation objects, as follows:		
\[
\xymatrix@C=70pt@R=30pt{
\txt<10pc>{Lie algebroid  \\$2$-term representation up to homotopy:\\ 
                           $A\circlearrowright  \E$\\ \ } \ar@{-->}[r]^{\txt{integration ?}}  
                                                    \ar@<-2ex>[d]_{\txt{\\transformation\\object\\ }}
                                                                                         & \txt<10pc>{ Lie groupoid \\$2$-term representation up to homotopy:\\ 
                                                                                                       $\G(A)\circlearrowright  \E$ \\ \ }  \ar@<-2ex>[d]_{\txt{\\transformation\\object\\}}\\
       \txt<10pc>{ \\ \vba :\\
                       $A\ltimes  \E=A\oplus E\oplus C$ } \ar[r]^{\txt{integration}} 
                                                     \ar@<-2ex>[u]_{\txt{splitting}}  
                                                                                         & \txt<10pc>{\\ \vbg: \\
                                                                                                       $\G(A\ltimes \E)\overset{?}{=}\G(A)\ltimes \E$}\ar@<-2ex>[u]_{\txt{horizontal\\lift}}}   \]

Unfortunately, unlike for usual representations, there are a few problems occuring:
\begin{enumerate}
\item even if $A$ is integrable, the \vba $A\ltimes \E$ may not be integrable,
\item even if $A\ltimes \E$ is integrable, a $2$-term representation up to homotopy of $A$ on $\E$ does not determine a $2$-term representation up to homotopy of $\G(A)$ on $\E$.
\end{enumerate} 

The first question is completely addressed in \cite{BCO}, showing explicit obstructions to the integrability of \vbas and hence providing integrability criteria for the integration of $2$-term representations up to homotopy of an integrable Lie algebroid.

The second item deserves more attention. Note first that, as shown in \cite{AriasSchatz}, any $2$-term representation up to homotopy of $\G(A)$ differentiates to a $2$-term representation up to homotopy of $A$. In this case, one can see that $\G(A)\ltimes \E$ integrates $A\ltimes \E$, both as a groupoid and as a \vbg \cite{BCO}. So there is no serious issue when performing the differentiation process. It is the integration of the splitting which is problematic. Indeed, given a \vbg over $\G(A)$, the corresponding $2$-term representation up to homotopy of $\G(A)$ is obtained by choosing a right-horizontal lift. Hence, given a $2$-term representation up to homotopy of $A$, once we have integrated $A\ltimes \E$ to $\G(A\ltimes \E)$, one still has to choose a right-horizontal lift of $\G(A\ltimes \E)$ (see \cite[Sec. 3]{GM10}) in order to obtain a $2$-term representation up to homotopy of $\G(A)$.

 We claim that such a right-horizontal lift is not determined by the infinitesimal data. Namely, the splitting of the \vba $A\ltimes \E=A\oplus E\oplus C$ does not determine a right-horizontal lift of the corresponding \vbg $\G(A\ltimes \E)$. It follows that one can not obtain a $2$-term representation up to homotopy of $\G(A)$ out of a $2$-term representation up to homotopy of $A$ without involving further choices. From this point of view, $2$-term representations up to homotopy of a Lie algebroid $A$ do not integrate to $2$-term representations up to homotopy of $\G(A)$, at least not in a canonical way.
 
 In \cite{CA01}, Arias Abad and Sch\"atz proposed an integration scheme for Lie algebroid representations up to homotopy of any degree, the output of which, when applied to the $2$-term case, is indeed not a $2$-term representation up to homotopy of a Lie groupoid, but rather an $A_\infty$-morphism of DG-categories. In their work, the notion of transformation object is not discussed and it is not clear how one can recover a representation up to homotopy of a Lie groupoid, or a \vbg, by this procedure.
 
 The aim of this paper is to provide an alternative integration-like scheme for a $2$-term representation up to homotopy of a Lie algebroid, which being more intuitive, makes it straightforward what the corresponding transformation object should be. We also derive from this construction the \vbg integrating a \vba in a natural way.

The integration of usual Lie algebroid representations via functors works as follows. Let $\G(A)$ denote the Weinstein groupoid of a Lie algebroid $A\to M$. A representation $\rho:A \longrightarrow  \mathrm{D}(E)$ of $A$ on a vector bundle $E\to M$ integrates to a morphism of groupoids
$$\hol:\G(A)\longrightarrow  \Gau(E),$$
which is just the holonomy of the flat $A$-connection $\nabla$ defined by $\rho$. In this way, we obtain a representation of $\G(A)$ on $E$, which can be used to describe an integration of the transformation Lie algebroid $A\ltimes E$.

In this paper we explain how a $2$-term representation up to homotopy of a Lie algebroid can be integrated via strict $2$-functors between $2$-groupoids. We use such a $2$-functor to construct the corresponding transformation $2$-groupoid and out of this object we derive a natural description of the Weinstein groupoid of the \vba associated to a $2$-term representation up to homotopy of a Lie algebroid. More applications are further discussed as well.

The organization of this paper is the following. In section 2 we briefly recall the notion of representation up to homotopy and its relation with both \vbas and \vbgs. Section 3 is devoted to the study of \vbas in connection with Lie algebroid extensions and fibrations. We show in Proposition \ref{prop:vbafibration} that any \vba $D$ gives rise to a Lie algebroid fibration. As a consequence, we prove Theorem \ref{prop:longexactfibrationvbacase} which describes a homotopy long exact sequence of a \vba and states that the monodromy groups of $D$ fit into an exact sequence:
 $$\SelectTips{cm}{}\xymatrix@C=15pt{\im (\delta_2,e)\  \ar@{^{(}->}@<-0.5pt>[r]& \MON(D,e)   \ar@{->>}@<-0.5pt>[r]^-{p} & \MON(A,m),}$$ 
 where $\delta_2$ is a transgression map associated with the underlying Lie algebroid fibration $D\to A$.

In section 4 we are concerned with $2$-groupoids and their representations. We introduce two natural examples of $2$-groupoids, the so called \textbf{Weinstein $2$-groupoid} $\tPa(A)$ of a Lie algebroid $A$ and the \textbf{gauge $2$-groupoid} $\tGau(\E)$ of a $2$-term complex of vector bundles $\E=(C\to E)$ over a fixed manifold. More precisely:
\begin{itemize}
\item  in $\tPa(A)$, objects are points in $M$, $1$-morphisms are thin homotopy classes of $A$-paths, and $2$-morphisms are homotopy classes of $A$-homotopies between $A$-paths.
\item  objects in $2$-$\Gau(\E)$ are points in $M$, $1$-morphisms $x\to y$ are given by \emph{invertible} chain maps $(C_x,E_x)\to (C_y,E_y)$, and $2$-morphisms are chain homotopies.
\end{itemize}

\noindent The main property of $\tPa(A)$ is presented in Proposition \ref{prop:weinsteintruncation} which says that the $1$-truncation of the Weinstein $2$-groupoid, i.e. the quotient space of $1$-morphisms by $2$-morphims, coincides with the Weinstein groupoid of $A$. Then, we introduce the notion of 2-representation or representation of a $2$-groupoid 2-$\G$, as a strict $2$-functor 2-$\G\to \tGau(\E)$. We state the first main result of this paper, which says that a $2$-term representation up to homotopy of a Lie algebroid $A$ can be integrated to a representation of the Weinstein $2$-groupoid of $A$. More precisely, we prove the following result.

\begin{theoremz}

\emph{A $2$-term representation up to homotopy of $A$ on $\E$ can be integrated into a strict $2$-functor}
\begin{align*}
\tPa(A)\overset{\hol}\longrightarrow \tGau(\E).
\end{align*}

\end{theoremz}

\noindent The proof of Theorem 1. can be found in Appendix A. The strict $2$-functor of Theorem 1 is referred to as the \textbf{holonomy $2$-representation}. Since a $2$-term representation up to homotopy of a Lie algebroid $A$ integrates to an action of $\tPa(A)$, rather than of the groupoid $\G(A)$, it is natural to consider the \emph{transformation $2$-groupoid} $\tPa(A)\ltimes \E$. The fact that $\tPa(A)\times \E$ is a 2-groupoid is the statement of Theorem \ref{thm:semidirect2groupoid}.

In section 5, we explain how the transformation $2$-groupoid $\tPa(A)\ltimes \E$ can be used in order to recover the \vbg integrating the \vba $D=A\ltimes \E$ underlying a $2$-term representation up to homotopy of a Lie algebroid $A$. This is the statement of the second main result of this work.

\begin{theoremz}
 \emph{Given a representation up to homotopy of a Lie algebroid $A$ on a $2$-term complex $\E$, the Weinstein groupoid of the associated \vba $D=A\ltimes \E$ identifies with the 1-truncation groupoid of the transformation 2-groupoid $\tPa(A)\ltimes \E$.}
\end{theoremz}

This point of view has both abstract and practical advantages. On the one hand, it allows us to interpret the integrability of a \vba as the vanishing of the second transgression map in the homotopy long exact sequence.  On the other hand, as we shall illustrate by various examples, the procedure can be implemented in order to obtain the \vbg of a \vba in a quite explicit way.

As a final note, the construction of the holonomy as a strict $2$-functor was inspired by the construction of Schreiber and Waldorf \cite{SW}, although their results would be hard to apply when the boundary map $\partial:C\to E$ does not have not constant rank. Also, the $2$-functoriality of the holonomy could probably be deduced from either \cite{SW} or \cite{CA01} by working leafwise, which we avoid using a direct proof. This also allows to compute explicitly the holonomy, which could be hard to deduce from \cite{SW} or \cite{CA01}.

Our approach is original as it is based on Lie algebroid extensions and fibrations, which we believe to be more natural for differential geometers. Furthermore, the construction of a strict transformation $2$-groupoid is original in this context. In fact, similar constructions are studied in \cite{SC, SC0, SC1} however, as we explain, they lead to different type of integrations that seem difficult to relate with \vbgs at first sight.

 \textbf{Acknowledgements:} O. Brahic would like to thank CNPq (Brazil), the PPGMA at UFPR, Curitiba, as well as IMPA, Rio de Janeiro, who supported both financially and logistically this project. C. Ortiz is thankful for financial support from CAPES-COFECUB (grant 763/13) and CNPq (grant 4827967/2013-8).

 The authors would also like to thank Matias del Hoyo and Pedro Frejlich for their comments in the early stages of this work.
\section{Generalities}

\subsection{$2$-term representations up to homotopy}
The notion of representation up to homotopy of a Lie algebroid was introduced in \cite{AriasCrainic}.
In our case of interest, namely when the underlying graded vector bundle is concentrated in degrees $-1$ and $0$,
the definition boils down to the following.

We consider a Lie algebroid $A$ over $M$, and $\E=(\partial:C\to E)$ a two term complex of vector bundles over $M$, concentrated in degree $-1$ and $0$. 

\begin{definition}\label{def:ruh}
A \textbf{2-term representation up to homotopy} of $A$ on $\mathcal{E}$ is a triple $(\nablaE,\nablaC,\omega)$ where $\nablaE,\nablaC$ are $A$-connections on $E$ and $C$ respectively, 
 and $\omega\in \Omega^2(A,\Hom(E,C))$, such that the following compatibility conditions are satisfied:
\begin{align}
\partial \circ \nablaC &= \nablaE\circ \partial,\label{def:ruh1}\\ 
                          \partial\circ \omega&=\omegaE,\\
                            \omega\circ \partial&=\omegaC,\\
\label{eq:ruh:omegaclosed} \nabla \omega&=0.
\end{align}
\end{definition}
Here $\omegaE$ and $\omegaC$ denote the respective curvatures of $\nablaE$ and $\nablaC$\!,
while $\nabla \omega\in\Omega^3(A,\Hom(E,C))$ denotes the covariant derivative naturally induced on $\Omega^\bullet(A,\Hom(E,C))$ by $\nablaE$ and $\nablaC$\!.
\begin{remark}
 In the original definition \cite{AriasCrainic}, the boundary operator $\partial:C\to E$ is part of the data that defines a representation up to homotopy.
 In this work, it is relevant to set the boundary operator apart and think of $A$ acting on a two term complex via $(\nablaE,\nablaC,\omega)$.
\end{remark}

There also exists a notion of representation up to homotopy of Lie groupoids \cite{AriasCrainic2, GM10}. In the case where the underlying graded vector bundle is concentrated in degrees $-1$ and $0$, it can be described as follows. 

Recall that any vector bundle $E\to M$ determines a smooth category $L(E)$ whose objects are elements of $M$, and morphisms between $x,y\in M$ are linear maps $E_x\to E_y$ (not necessarily invertible). Given a Lie groupoid $\G\tto M$, a \textbf{unital quasi-action} of $\G$ on $E$ is a smooth map $\Delta: \G\to L(E)$ that commutes with the unit, source and target maps. Note that $\Delta$ does not necessarily commute with the multiplication. When it is the case, the unital quasi action is said to be \textbf{flat}. Also, $\Delta_g:E_{\s(g)}\to E_\t(g)$ is not required to be invertible. A \textbf{representation} of $\G$ on $E$ is just a flat unital quasi-action $\Delta:\G\to L(E)$ such that $\Delta_g:E_{\s(g)}\to E_{\t(g)}$ is invertible for each $g\in \G$. 

Let us now fix a Lie groupoid $\G\tto M$, and $\E=(\partial:C\to E)$ a two term complex of vector bundles over $M$, concentrated in degree $-1$ and $0$. We denote by $\G^{(2)}=\{(g,h): \s(g)=\t(h)\}$ the space of composable arrows of $\G$.

\begin{definition}\label{def:ruthgroupoids}
 A \textbf{$2$-term representation up to homotopy} of $\G$ on $\E=(\partial:C\to E)$ is given by a triple $(\DeltaC,\DeltaE,\Omega)$, where $\Delta^C$ and $\DeltaE$ are unital quasi-actions on $C$ and $E$ respectively, and $\Omega \in C^\infty(\G^{(2)}, \mathrm{Hom}(E,C))$ is a normalized cochain (i.e. $\Omega_{g_1,g_2}=0$ if either $g_1$ or $g_2$ is a unit), satisfying the following conditions:
\begin{align}
\partial\circ \DeltaC_{g_1}   =\DeltaE_{g_1}\circ \partial,\\
\Omega_{g_1,g_2}\circ \partial =\DeltaC_{g_1g_2}- \DeltaC_{g_2}\circ \DeltaC_{g_1},\\
\partial\circ\Omega_{g_1,g_2}=\DeltaE_{g_1g_2}- \DeltaE_{g_2}\circ \DeltaE_{g_1},\\
\DeltaC_{g_1}\circ \Omega_{g_2,g_3}-\Omega_{g_1g_2,g_3}+\Omega_{g_1,g_2g_3}-\Omega_{g_1,g_2}\circ \DeltaE_{g_3}=0,
\end{align}

\begin{align}
\partial\circ \DeltaC_{g_1}   &=\DeltaE_{g_1}\circ \partial,\\
\Omega_{g_1,g_2}\circ \partial &=\DeltaC_{g_1g_2}- \DeltaC_{g_2}\circ \DeltaC_{g_1},\\
\partial\circ\Omega_{g_1,g_2}&=\DeltaE_{g_1g_2}- \DeltaE_{g_2}\circ \DeltaE_{g_1},\\
\Omega_{g_1g_2,g_3}-\Omega_{g_1,g_2g_3}
&=\DeltaC_{g_1}\circ \Omega_{g_2,g_3}-\Omega_{g_1,g_2}\circ \DeltaE_{g_3},
\end{align}

for every triple $g_1,g_2,g_3\in \G$ where the above multiplications make sense.
\end{definition}

\subsection{\texorpdfstring{\vbgs}{VB-groupoids}}
Closely related with $2$-term representations up to homotopy of Lie algebroids and groupoids, are the notions of \vba and \vbg, which we briefly recall. Detailed expositions can be found in \cite{mackenzie-book, GM08, GM10}.

\begin{definition} A \textbf{$\mathcal{VB}$-groupoid} is a square
\begin{align}\label{VBgroupoid}
\SelectTips{cm}{}\xymatrix{ \HH \ar@<0.5ex>[d]\ar@<-.5ex>[d]\ar[r]^{q_H} &\G\ar@<0.5ex>[d]\ar@<-.5ex>[d]\\
           E \ar[r]^{q_E}& M}
\end{align}
where double arrows denote Lie groupoid structures and single arrows denote vector bundles. It is required that the structure mappings (source, target, multiplication, unit section and inversion) that define the Lie groupoid $\HH\rightrightarrows E$ be morphisms of vector bundles over the corresponding structure mappings defining the Lie groupoid $\G\rightrightarrows M$.
\end{definition}

 A $\mathcal{VB}$-groupoid as in \eqref{VBgroupoid} will be denoted by $(\HH;\G,E;M)$.

\begin{example}[Tangent groupoid]
Let $\G\rightrightarrows M$ be a Lie groupoid. The tangent bundle $T\G$ has a canonical Lie groupoid structure over $TM$. The structural maps of $T\G\rightrightarrows TM$ are defined by applying the tangent functor to each of the structural maps of $\G\rightrightarrows M$. It can be easily checked that with respect to these maps, the quadruple $(T\G;\G,TM;M)$ is a $\mathcal{VB}$-groupoid, referred to as the \textbf{tangent groupoid} of $\G$.
\end{example}

\begin{example}[Cotangent groupoid]

Given a Lie groupoid $\G\rightrightarrows M$ with Lie algebroid $A$, the cotangent bundle $T^*\G$ is equipped with a groupoid structure over $A^*$. An explicit description of this groupoid can be found also in \cite{CDW}. The source and target maps are defined by
$$\tilde{s}(\alpha_g)u=\alpha_g(Tl_g(u-Tt(u)))\quad \text{ and }\quad \tilde{t}(\alpha_g)v=\alpha_g(Tr_g(v))$$
where $\alpha_g \in T^*_g\G$, $u\in A_{s(g)}\G$ and $v\in A_{t(g)}\G$. The multiplication on $T^*G$ is defined by
$$(\alpha_g\circ \beta_h)(X_g\bullet Y_h)= \alpha_g(X_g)+ \beta_h(Y_h)$$
for $(X_g,Y_h)\in T_{(g,h)}\G_{(2)}$. We refer to $T^*\G$ with the groupoid structure over $A^*$ as the \textbf{cotangent groupoid} of $\G$. 
\end{example}

\subsubsection*{\vbgs and $2$-term representations up to homotopy of Lie groupoids}

 Given a $2$-term representation up to homotopy $(\Delta^E,\Delta^C, \Omega)$ of a Lie groupoid $\s,\t:\G\tto E$ on $\E=(\partial:C\to E)$, one can construct a \vbg $(\G\ltimes E;\G,E;M)$ in the following way \cite{GM10}. Consider the direct sum of the pull back bundles $\G\ltimes E:=\t^*C\oplus \s^*E$, which is a vector bundle over $\G$. Then $\G\ltimes E$ comes equipped with a groupoid structure over $E$ whose source, target and identity maps are given by
$$\tilde{s}(c,g,e)=e, \quad \tilde{t}(c,g,e)=\partial(c)+\Delta^E_g(e), \quad \tilde{1}(e)=0\oplus e,$$
and with the following groupoid multiplication:
$$(c_1,g_1,e_1)\cdot (c_2,g_2,e_2)=\left(c_1+\Delta^C_{g_1}(c_2)-\Omega_{g_1,g_2}(e_2),g_1\cdot g_2,e_2\right).$$
One may think of $\G\ltimes E\tto E$ as a \emph{transformation groupoid} associated with the $2$-term representation up to homotopy $(\Delta^E,\Delta^C, \Omega)$.

Conversely, given a \vbg $(\HH,\G,E,M)$, the \textbf{core} bundle is the vector bundle $C:=\mathrm{Ker}(s_{\HH})|_{M}$. The \textbf{core sequence} of $\HH$ is the exact sequence of vector bundles over $\G$

$$0\to \t^*C\to \HH\to \s^*E\to 0,$$

\noindent where the first map is $(g,c_{\t(g)})\mapsto c_{\t(g)}0^{\HH}_g$ and the second map is the bundle map induced by the source $\s_{\HH}:\HH\to E$. A \textbf{right horizontal lift} of $\HH$ is defined in \cite[def.\,3.8]{GM10} as a right splitting of the core sequence, which coincides with the canonical splitting $\HH_{x}=C_x\oplus E_x$ along any unit $x\in \G$.  This allows one to identify $\HH$ with a \vbg of the form $\t^*C\oplus \s^*E$ as above for some $2$-term representation up to homotopy $(\Delta^E,\Delta^C, \Omega)$. Right horizontal lifts always exist and we refer to \cite{GM10} for more details. 
\subsection{\texorpdfstring{\vbas}{VB-algebroids}}\label{sec:vba}
Let us recall some definitions regarding double vector bundles and their special sections.
We refer to \cite{KU, mackenzie-book, Pradines} for a more detailed treatment.

\begin{definition}
A \textbf{double vector bundle} (DVB) is a commutative square
\begin{align}\label{doubleVB}
\SelectTips{cm}{}
\xymatrix{ D \ar[d]_{p_D}\ar[r]^{p} &A\ar[d]^{p_A}\\
           E \ar[r]^{\pE}& M}\end{align}
where all four sides are vector bundles, $p_D$ is a vector bundle morphism over $p_A$, and where $+_{E}: D\times_E D \longrightarrow D$ is a vector bundle morphism over $+: A\times_M A \longrightarrow A$. Here $+_{E}$ is  the addition map for the vector bundle $D\longrightarrow E$.
\end{definition}

In other words, the quadruple $(D;A,E;M)$ is a $\mathcal{VB}$-groupoid where $D\longrightarrow E$
and $A\longrightarrow M$ are equipped with the Lie groupoid structure induced by the corresponding vector bundle structures. 

Given a DVB $(D; A, E; M)$, the vector bundles $A$ and $E$ are called
the \textbf{side bundles}. The zero sections are denoted by $0^{A}: M \longrightarrow A$, $0^{E}: M \longrightarrow
E$, $0^D_{A}: A \longrightarrow D$ and $0^D_{E}: E
\longrightarrow D$. Elements of $D$ are written $(d; a, e; m)$,
where $d \in D$, $m\in M$ and $a=p(d) \in A_m$, $e=p_D(d)\in
E_m$.

The \textbf{core} $C$ of a DVB is the intersection of the kernels of
$p$ and $p_D$. It has a natural vector bundle structure over
$M$, the projection of which we call $\pC: C \longrightarrow M$. The
inclusion $C \hookrightarrow D$ is usually denoted by
$$
C_m \ni c \longmapsto \overline{c} \in p^{-1}(0^A_{m}) \cap p_{D}^{-1}(0^E_{m}).
$$

Given a DVB $(D;A,E;M)$, the space of sections $\Gamma(E, D)$ is
generated as a $C^{\infty}(E)$-module by two distinguished classes of
sections, namely linear sections and core sections, which we now describe.

\begin{definition}\label{def:coresections}\em
  For a section $c: M \rightarrow C$, the corresponding \textbf{core
    section} $\hat{c}: E \rightarrow D$ is defined as
\begin{equation}\label{core_section}
\hat{c}(e_m) = 0^D_{E}(e_m) +_{A} \overline{c(m)}, \,\, m \in M, \, e_m \in E_m.
\end{equation}
\end{definition}

We denote the space of core sections by $\Gamma_c(E, D)$.

\begin{definition}\em
  A section $\chi \in \Gamma(E,D)$ is called \textbf{linear} if $\chi: E
  \rightarrow D$ is a vector bundle morphism covering a section $a:M\longrightarrow A$. The space of linear sections is denoted by  $\Gamma_{\ell}(E, D)$.
\end{definition}

\begin{definition}\label{def:vba}

A \textbf{$\mathcal{VB}$-algebroid} is a double vector bundle $(D;A,E;M)$ as follows:
\begin{equation}\label{vbdiagram}
\SelectTips{cm}{}
\xymatrix{ D \ar[d]_{p_D}\ar[r]^{p} &A\ar[d]^{p_A}\\
           E \ar[r]^{\pE}& M,}
\end{equation}
where $D\longrightarrow E$ and $A\longrightarrow M$ are equipped with Lie algebroid structures, such that the anchor map 
$\rho_D: D \longrightarrow TE$ is a bundle morphism over $\rho_A: A \longrightarrow TM$ and the following bracket conditions are satisfied:
\begin{itemize}
\item[i)] $[\Gamma_{\ell}(E, D), \Gamma_{\ell}(E,D)]_D \subset \Gamma_{\ell}(E,D)$;
\item[ii)] $[\Gamma_{\ell}(E,D), \Gamma_c(E,D)]_D \subset \Gamma_c(E,D)$;
\item[iii)] $[\Gamma_c(E,D), \Gamma_c(E,D)]_D = 0$.
\end{itemize}
\end{definition}

\begin{example}
 Given a Lie algebroid $A\longrightarrow M$, the application of the tangent functor to each of the structural maps
 defining the vector bundle $A\longrightarrow M$, determines a double vector bundle $(TA;A,TM;M)$. Moreover, the Lie algebroid structure of $p_A:A\longrightarrow M$ can be lifted to $Tp_A:TA\longrightarrow TM$, making the quadruple $(TA;A,TM;M)$ into a $\mathcal{VB}$-algebroid called the \textbf{tangent algebroid} of $A$. For more details, see \cite{mackenzie-book}.
\end{example}

\begin{example}
Given a Lie algebroid $A\to M$, the cotangent bundle $T^*A$ inherits a Lie algebroid structure over $A^*$, yielding a $\mathcal{VB}$-algebroid $(T^*A;A;A^*,M)$ referred to as the \textbf{cotangent algebroid} of $A$. For more details, see \cite{mackenzie-book} and the references therein. 

\end{example}



\subsubsection*{\vbas and 2-term representations up to homotopy of Lie algebroids}

This section follows closely \cite{GM08}. Let $(D;A,E;M)$ be a $\mathcal{VB}$-algebroid. The space of linear sections $\Gamma_\ell(E,D)$ is locally free as a $C^{\infty}(M)$-module.
Hence, there exists a vector bundle $\hat{A}\to M$ whose space of sections $\Gamma(\hat{A})$ identifies with $\Gamma_\ell(E,D)$ as  a $C^{\infty}(M)$-module. Also, every linear section $\chi:E\to D$ covers a section $a:M\to A$, inducing an exact sequence of vector bundles over $M$
\begin{equation}\label{eq:exactVB}
\SelectTips{cm}{}\xymatrix{ \Hom(E,C) \ \ar@{^{(}->}@<-0.15pt>[r] &\hat{A}  \ar@{->>}[r]<-0.15pt> & A.}
\end{equation}

One observes that $\hat{A}\to M$ is a Lie algebroid with bracket $[\chi_1,\chi_2]_{\hat{A}}=[\chi_1,\chi_2]_D$ and anchor map $\rho_{\hat{A}}(\chi)=\rho_A(a)$, for every $\chi\in\Gamma(\hat{A})$ linear section covering $a\in\Gamma(A)$. Also, the vector bundle $\mathrm{Hom}(E,C)\to M$ is equipped with a Lie algebroid structure zero anchor map and bracket $[\phi,\psi]=\phi\circ \partial \circ \psi - \psi \circ \partial \circ \phi$, where $\partial:C\to E; c\mapsto \rho_D(c)$ is the \textbf{core anchor} of $D$. Thus, the sequence \eqref{eq:exactVB} is an extension of Lie algebroids over $M$.

\begin{example}
Given a Lie algebroid $A\to M$, one can consider the tangent \vba $(TA;A,TM;M)$. In this case, the exact sequence \eqref{eq:exactVB} reads:
$$\SelectTips{cm}{}\xymatrix{ \Hom(TM,A) \ \ar@{^{(}->}@<-0.25pt>[r] &\mathfrak{J}(A)  \ar@{->>}[r]<-0.25pt> & A,}$$
where $\mathfrak{J}(A)\to M$ is the first jet algebroid associated to $A$.
\end{example}

There exists a natural representation of $\hat{A}$ on $C$:
\begin{equation}\label{AhatC}
\Phi_{\chi}(c):=[\chi,\hat{c}]_D.
\end{equation}
Similarly, there is a representation of $\hat{A}$ on $E$, defined as follows. Since $\rho_D:D\to TE$ is a DVB-morphism, for each linear section $\chi\in\Gamma_l(E,D)$, the vector field $\rho_D(\chi)\in\mathfrak{X}(E)$ is linear, and hence $\mathcal{L}_{\rho_D(\chi)}(\xi)\in \Gamma(E^*)$ for every $\xi\in\Gamma(E^*)\simeq C^{\infty}_{lin}(E)$, viewed as a fiberwise linear function on $E$. Thus, we obtain a representation of $\hat{A}$ on $E$ by the formula
\begin{equation}\label{AhatE}
\langle \xi, \Psi_{\chi}(e)\rangle = \rho_{\hat{A}}(\chi)\langle \xi,e\rangle - \langle \mathcal{L}_{\rho_D(\chi)}\xi, e\rangle,
\end{equation}

\noindent for every $\chi\in\Gamma_l(E,D)$, $\xi\in\Gamma(E^*)$ and $\Gamma(e)$.

Now, following \cite{GM08}, we proceed to briefly explain the relation between \vbas over $A$ and $2$-term representations up to homotopy of $A$. Let $(D;A,E;M)$ be a $\mathcal{VB}$-algebroid. This determines a $2$-term complex $\partial:C\to E$ defined by the core anchor map. A splitting $h:A\to \hat{A}$ of the sequence \eqref{eq:exactVB} is referred to as a \textbf{horizontal lift}. Notice that horizontal lifts are not required to be morphisms of Lie algebroids. The failure for $h$ being a Lie algebroid morphism is controlled by the element $\omega\in \Gamma(\wedge^2A^*\otimes \mathrm{Hom}(E,C))$ defined by
\begin{equation}\label{curvaturehorlift}
\omega(a,b):=h([a,b]_A) - [h(a),h(b)]_{D}.
\end{equation}

The choice of a horizontal lift $h:A\to\hat{A}$ defines $A$-connections on $E$ and $C$,
by pulling back via $h:A\to \hat{A}$ the representations \eqref{AhatC} and \eqref{AhatE}, respectively.
That is, the $A$-connection $\nablaC$ on $C$ is given by
\begin{equation}\label{AconnC}
\nablaC_ac:=[h(a),\hat{c}]_{D},
\end{equation}

\noindent and similarly, the $A$-connection on $E$ is defined by
\begin{equation}\label{AconnE}
\nablaE_ae:=\Psi_{h(a)}e,
\end{equation}
where $\Psi$ as in \eqref{AhatE}. These $A$-connections are not flat unless $\omega=0$.

As explained in \cite{GM08}, the $A$-connections $\nablaE,\nablaC$ together with the element $\omega\in \Gamma(\wedge^2A^*\otimes \mathrm{Hom}(E,C))$ given by \eqref{curvaturehorlift}, define a representation up to homotopy of the Lie algebroid $A$ on the $2$-term complex $\partial:C\to E$. See Definition \ref{def:ruh}. Conversely, a representation up to homotopy $(\nabla^E,\nabla^C,\omega)$ of a Lie algebroid $A$ on a $2$-term complex $\partial:C\to E$ induces a Lie algebroid structure on the split double vector bundle $D=A\times_M\times C\times_M E$. This is summarized in the following result.

\begin{theorem}{(\cite{GM08})}\label{lma:conn data} Let $A\to M$ be a Lie algebroid and $E, \ C$ vector bundles over $M$. The  formulas \eqref{curvaturehorlift}, \eqref{AconnC} and \eqref{AconnE}  define a $1:1$ correspondence between split \vba structures $D=A\times_M C\times_M E$ with core anchor $\partial:C \to E$ and representations up to homotopy $(\nabla^E,\nabla^C,\omega)$ of $A$ on the complex $\partial:C\to E$ as above.
\end{theorem}

The notion of morphism between representations up to homotopy can be found in \cite{GM08}.
As we already observed, given a \vba $(D;A,E;M)$ with core bundle $C$, then the choice of a horizontal lift $h:A\to \hat{A}$ determines a 2-term representation up to homotopy on $C\oplus E$. Moreover, isomorphism classes of \vba structures on the DVB $(D;A,E;M)$ with core $C$, are in one to one correspondence with isomorphism classes of representation up to homotopy of $A$ on $C\oplus E$ (see e.g. \cite{GM08}). Indeed, this correspondence is given by an equivalence of categories \cite{DJO}.

\subsection{\texorpdfstring{\vbgs}{VB-groupoids} vs \texorpdfstring{\vbas}{VB-algebroids}}
As explained in \cite{mackenzie-book}, given a $\mathcal{VB}$-groupoid $(\HH;\G,E;M)$,
the application of the Lie functor to each of the groupoids $\HH\rightrightarrows E$ and $\G\rightrightarrows M$,
yields a $\mathcal{VB}$-algebroid $(A\HH;A\G,E;M)$.
In this case, we say that the $\mathcal{VB}$-groupoid $(\HH;\G,E;M)$ integrates the $\mathcal{VB}$-algebroid $(A\HH;A\G,E;M)$.

\begin{definition}\em
A $\mathcal{VB}$-algebroid $(D;A,E;M)$ is called \textbf{integrable} if there exists a $\mathcal{VB}$-groupoid $(\HH;\G,E;M)$ such that $(A\HH;A\G,E;M)$ is isomorphic to $(D;A,E;M)$.
\end{definition}

The following result can be found in \cite{BCdH}.

\begin{theorem}{\cite{BCdH}}\label{thm:DintegGinteg}
Let $(D;A,E;M)$ be a $\mathcal{VB}$-algebroid. Assume that $D\to E$ integrates to a source simply connected Lie groupoid $\G(D)\rightrightarrows E$. Then, the Lie algebroid $A\to M$ is integrable, and $(\G(D);\G(A),E;M)$ is a $\mathcal{VB}$-groupoid integrating $(D;A,E;M)$, where $\G(A)$ is the source simply connected integration of $A$.
\end{theorem}

Theorem \ref{thm:DintegGinteg} says essentially that whenever $D$ is integrable, the \vbg structure on $\G(D)$ comes for free. Note however that it says nothing about how to decide whether a given \vba algebroid is integrable or not. On that matter, the following was proved in \cite{BCO}.

\begin{theorem}{\cite{BCO}}\label{thm:intiff} Let $(D,A,E,M)$ be a \vba.
 Then, $D$ is integrable if and only if A is integrable and $\Mon(D)\cap \ker p=\{0^D_E\}$. Here  $\mathcal{N}(D)\subset D$ denotes the monodromy groups of $D$, $p:D\to A$ the \vba projection, and $0^D_E$ the zero section of $D\to E$.
\end{theorem}

Note that both results are only concerned about \vbas and their integration to \vbgs. If one is rather interested in $2$-term representations up to homotopy, one may wonder if it is possible to integrate a $2$-term representation up to homotopy of a Lie algebroid to a $2$-term representation up to homotopy of a Lie groupoid in a canonical way (meaning, freely of choices). Surprisingly, the answer is \emph{no} (see Remark \ref{rem:Rajehorizontalliftexample}) which was one of the motivations for this work. We will be back to this later.

\section{\texorpdfstring{\vbas}{VB-algebroids} as Lie algebroid fibrations}

\subsection{Generalities on Lie algebroid fibrations}
We now briefly recall the notions of extension and fibration for Lie algebroids  \cite{Br, BZ}.

\begin{definition}
 A Lie algebroid {\bf extension} is a surjective Lie algebroid morphism $\SelectTips{cm}{}\xymatrix@C=15pt{p:A_E  \ar@{->>}[r] & A_{}}$ covering a surjective submersion $p_E:E \to M$.
\end{definition}

The \textbf{kernel} of an extension $\SelectTips{cm}{}\xymatrix@C=15pt{p:A_E  \ar@{->>}[r] & A},$ is defined by $\K:=\ker p$. One can show \cite{Br} that $\K$ is a Lie subalgebroid of $A_E$ whose orbits are included in the fibers of the base submersion $p_E:E\to M$. Let us stress the fact that $\K$ coincides with the kernel of $p$ as a vector bundle map, in particular $\K$ is a Lie algebroid over $E$. Therefore, taking into account base manifolds, we obtain a short exact sequence of Lie algebroids
\begin{gather}\label{eq:extension}
\SelectTips{cm}{}\xymatrix{ \K_{} \ \ar@{^{(}->}@<-1pt>[r]^{i} \ar[d]&A_E  \ar@{->>}[r]<-1pt>^{p} \ar[d]& A_{}\ar@{->}[d] \\
             E \ar@{->}[r]^{\id}&E  \ar@{->}[r]^{p_E}              & M.}
\end{gather}
\begin{definition}
Given a Lie algebroid extension, an {\bf Ehresmann connection} is given by a vector subbundle $H\to E$ complementary to $\K$ in $A_E$, that is $\K\oplus H=A_E$.  
\end{definition}

As easily seen, a connection is characterized by a $C^\infty(M)$-linear map $\hor:\Gamma(A)\to \Gamma(H)$ referred to as the \textbf{horizontal lifting}. This determines a $C^{\infty}(M)$-linear map $\mathcal{D}:\Gamma(A)\to \Der(\K)$, called \textbf{covariant derivative}, and defined by $\mathcal{D}_a:=[\hor(a),\cdot ]_{A_E}$. Then the {\bf curvature} of the connection is the element $\omk\in \Omega^2(A, \Gamma(\K))$, defined by $\omk(a,b)=\hor([a ,b ]_A)-[\hor(a),\hor(b) ]_{\scriptscriptstyle{{A_E}}}.$

The couple $(\mathcal{D},\omk)$ entirely determines the extension and, as a consequence of the Jacobi identity on $A_E$, it satisfies the following conditions:
\begin{gather}
\mathcal{D}_{[a,b]_{A}}-[\mathcal{D}_a,\mathcal{D}_b]=[\omk(a,b), \ ]_{\scriptscriptstyle{\K}},\\
\oint_{a,b,c}\mathcal{D}_a\,\omk(b,c)+\omk([a,b]_A,c) = 0.
\end{gather}

\begin{definition}
An Ehresmann connection is  {\bf complete} if for any section $a\in\Gamma(A)$ such that $\rho_A(a)$ is a complete vector field, then $\rho_{A_E}(\hor(a))$ is a complete vector field as well. 
\end{definition}

For a complete connection, the notion of parallel transport along $A$-paths is well defined \cite{Br}.

\begin{definition}
Let $A\to M$ be a Lie algebroid. An {\bf $A$-path} is a Lie algebroid morphism $adt:TI\to A$, where $a:I\to A$ and $I=[0,1]$.
\end{definition}

Given an $A$-path $a:I\to A$ between $m_0,m_1\in M$, the corresponding holonomy is a Lie algebroid morphism $\hol(a):\K|_{E_{m_0}}\!\to \K|_{E_{m_1}},$ where $E_{m_i}:=p^{-1}_E(m_i)$ denotes the fiber over $m_i$. The presence of a complete connection also allows one to lift $A$-paths to \emph{horizontal} $A_E$ -paths. In \cite{BZ}, this lifting path property motivated the following definition.
\begin{definition}
A Lie algebroid {\bf fibration} is an extension $\SelectTips{cm}{}\xymatrix@C=15pt{p:A_E  \ar@{->>}[r] & A_{}}$ that admits a complete Ehresmann connection.
\end{definition}

 The fundamental groups $\pi_\bullet(A)$ of a Lie algebroid are obtained by taking $A$-spheres up to homotopy, where both $A$-spheres and homotopies are defined as Lie algebroid maps $TI^n\to A$ with certain boundary conditions (see the Section \ref{sec:Weinsteintwogroupoid} ). Recall that in general, $\pi_n(A)$ is a bundle of abelian groups over $M$. The main result of \cite{BZ} establishes a homotopy long exact sequence for Lie algebroid fibrations as follows.

\begin{theorem}[{\cite{BZ}}]\label{thm:longexactfibration}
Given a Lie algebroid fibration $\SelectTips{cm}{}\xymatrix@C=15pt{p:A_E  \ar@{->>}[r] & A_{}}$ with kernel $\K$, there are transgression maps $\delta_n:\pE^*\pi_{n}(A)\to \pi_{n-1}(\K)$ that fit into a homotopy long exact sequence of groupoids:
\begin{align*} \cdots \xrightarrow{}  \pi_{n}(\K) \xrightarrow{i} \pi_{n}(A_E) \xrightarrow{p}\pE^*\pi_{n}(A) &\xrightarrow{\delta_{n}}\pi_{n-1}(\K)\xrightarrow{} \cdots \\
   \cdots \xrightarrow{} \pE^*\pi_{2}(A)&\xrightarrow{\delta_2  } \pi_1(\K) \xrightarrow{i} \pi_1(A_E) \stackrel{p}{\twoheadlongrightarrow} \pi_1(A).
\end{align*}
\end{theorem}

Here, $\pE^*\pi_{n}(A)$ denotes the pull back bundle of groups along the projection $\pE:E\to M$. This pull back comes into play when lifting $A$-spheres to $A_E$-spheres, since a base point in $E$ needs to be chosen.


\subsection{\texorpdfstring{\vbas}{VB-algebroids} and Lie algebroid extensions}

We now consider a \vba $(D;A,E;M)$ as in \eqref{vbdiagram}. Since the underlying map $p: D\to A$ is a surjective Lie algebroid morphism, covering  a surjective submersion $p_E:E\to M$, we see that a \vba is a special case of a Lie algebroid extension.

\begin{proposition}\label{seq:extension-vba}
Any \vba  $(D;A,E;M)$ with core $C$ defines a Lie algebroid extension $\SelectTips{cm}{}\xymatrix@C=15pt{p:D  \ar@{->>}[r]^{} & A_{}}$
whose kernel $\K\to E$ is canonically identified with $\pE^*C \to E$ as a vector bundle.
\end{proposition}

\begin{proof}
By the definition of the core (section \ref{sec:vba}) we see that $C$ injects into $\K$ as the restriction to the zero section of $\pE:E\to M$,
namely $C=\K|_{0^E}$. Then one can use the addition $+_A:D\times_A D\to D$ for the vector bundle structure on $D\to A$ to obtain an isomorphism $\K=\pE^*C$. Note that this identification is canonical in the sense that it does not depend on the choice of a splitting of $D$.
\end{proof}

\begin{remark}
Note that the extension of Proposition \ref{seq:extension-vba} is quite different from \eqref{eq:exactVB}. For instance, $D$ is an extension of $A$ with different base spaces \eqref{eq:extension} which will play an important role in this work.
\end{remark}

The above result applies in particular to the differential $d\pE:TE\to TM$ of the projection of an arbitrary vector bundle $p_E:E\to M$. In that case we recover the following well known identification.
\begin{proposition}\label{seq:extension-vba-tangent}
For any vector bundle $\pE:E\to M$, the vertical bundle $\ker d\pE\to E$, where 
$\SelectTips{cm}{}\xymatrix@C=15pt{dp_E: TE  \ar@{->>}[r]& TM_{}}$ denotes the differential of the projection, 
						 identifies with $\pE^* E \to E$ as a vector bundle.
\end{proposition}

\subsection{The kernel of a \vba}
The kernel of an extension induced by a \vba enjoys interesting properties. In order to explain this, we denote by $\Vec_M$ the category of real smooth vector bundles over a smooth manifold $M$, where morphisms are smooth vector bundle maps covering the identity of $M$. In the following, we adapt the discussion from \cite{BaezCrans6}.

\begin{definition}
A \textbf{$2$-vector bundle} over $M$ is a category internal to $\Vec_M$.
\end{definition}

Namely, a $2$-vector bundle is a category $E_1\tto E_0$ where both the space of objects $E_0$ and of arrows $E_1$ are vector bundles over $M$, and where all structure maps (source, target, unit and composition) are vector bundle morphisms covering the identity. Morphisms between $2$-vector bundles over $M$ are \textbf{linear functors}, that is, functors which are vector bundle maps both at the level of objects and of arrows. In this way, we obtain a category $\tVec_M$ of $2$-vector bundles over $M$. We also denote by $\tTerm_M$ the category of $2$-term complexes of vector bundles over $M$.
 
\begin{proposition}\label{Dold-Kan}
There is an equivalence of categories $\mathrm{DK}:\tVec_M \to\tTerm_M.$ 
\end{proposition}
\begin{proof}
As explained in \cite{BaezCrans6}, this equivalence is an instance of the Dold-Kan correspondence. We only sketch the proof and refer to \cite{BaezCrans6} for more details. To any $2$-vector bundle
$E_1\tto E_0$, one can associate a $2$-term complex of vector bundles over $M$ given by $\E:=(\d:C\xrightarrow{} E_0)$, where $C:=\mathrm{Ker}(\s)$ and $\partial:=\t|_C$. Conversely, given a $2$-term complex $\E=(\d:C\to E)$ of vector bundles over $M$, we associate a $2$-vector bundle $
 C\oplus E\tto E$, where an arrow $(c,e)$ has source $\s(c,e)=e$ and target $\t(c,e)=e+\dC c$. The multiplication of composable arrows is given by $(c_2,e+\dC c_1 )\cdot(e,c_1)=(c_1+c_2,e)$,
while units are given by ${\bf 1}_e=(0,e)$. Both correspondences are functorial and induce an equivalence of categories (see  \cite{BaezCrans6}).
\end{proof}

As a consequence of Proposition \ref{Dold-Kan} we conclude that a $2$-vector bundle is necessarily a Lie groupoid. The inversion map $C\oplus E\to C\oplus E$ is defined by $(c,e)^{\bf -1}=(-c,e+\dC c)$

The next result gives a first hint of the relevance of treating a \vba as a Lie algebroid extension with different bases as in \eqref{eq:extension} rather than looking at the corresponding jet Lie algebroid \eqref{eq:exactVB}.

\begin{proposition}\label{prop:int:kernel}
Let $(D;A,E;M)$ be a \vba. The kernel $\K$ of the induced extension integrates to a $2$-vector bundle
whose associated $2$-term complex is $\E=(\d:C\to E)$.
\end{proposition}

\begin{proof}
First we use Proposition \ref{seq:extension-vba} to identify $\K$ with $\pE^* C=C\oplus E$. In this way, core sections are exactly the sections of $\K$ of the form $\hat{c}:e\mapsto (c\circ \pE(e),e)$ for some $c\in\Gamma(C)$. Note that they span $\Gamma(\K)$ as a $C^\infty(E)$-module, so they entirely determine the Lie algebroid structure on $\K$. Furthermore, we easily deduce from condition $iii)$ in Definition \ref{def:vba}, and the discussion in Section \ref{sec:vba} that:
\begin{align*}
 \rho_{\scriptscriptstyle{\K}}(c_1,e)&=(\partial c_1,e),\\
 [\hat{c}_1,\hat{c}_2]_\K&=0,
\end{align*}
for any $c_i\in \Gamma(C)$. In the above equation,  $(\partial c_1,e)\in \pE^* E$, where $\pE^* E=E\oplus E$ is identified with $\ker d\pE\subset TE$ by Proposition \ref{seq:extension-vba-tangent}. As easily checked, we recover the bracket and anchor of the Lie algebroid associated with $\pE^* C\tto E$ as in the proof of the Proposition \ref{Dold-Kan}. Since the latter is source-simply connected, we conclude that $\G(\K)=\pE^* C\tto E$.
\end{proof}

\begin{remark}\label{rem:int:kpaths}
The following gives an explicit integration procedure for $\K$-paths that will be useful later on. Given a $\K$-path $ads:I\to \K$, where $a(s)=(c(s),e(s))$,
there is an obvious way to extend $a$ into an $s$-dependent section of $\K$, at least fiber-wise, simply by taking the corresponding time-dependent core section. Since core sections commute, similarly to the integration of an abelian Lie algebra, it is easily checked that $a$ is always $\K$-homotopic to a ``straight" path, namely: $s\mapsto ( \overline{c},e(0)+s\partial(\overline{c}))$, where $\overline{c}$ is given by $\overline{c}:=\int_0^1 c(s)ds.$ It follows that the quotient map by
$\K$-homotopies can be explicitly written as an integral:
\begin{align}
    P_1(\K)=P_1(\pE^* C)   &\longrightarrow \G(\K)=\pE^* C
    \notag \\
\bigl(c(s),e(s)\bigr)ds &\longmapsto \bigl(\int_0^1 c(s)ds,e(0) \bigr).\end{align}
\end{remark}

\subsection{Relating connections and curvatures}
We now explain how a connection on a \vba induces an Ehresmann connection on the underlying extension and how their respective horizontal lifts, curvatures and holonomy are related.

Given a split \vba $(D=A\oplus E \oplus C;A,E;M)$, there is a natural Ehresmann connection whose horizontal subbundle is given by $H=A\oplus E\to E$. Then necessarily, the horizontal lifting map $\Gamma(A)\to \Gamma(H)$ has values in linear sections of $D$, coinciding with the map $A\to \hat{A}\subset\Gamma_\ell(E,D),\ a\mapsto \hor(a)$. Also it follows that the curvatures of a such a connection, seen both as a linear connection on a \vba and as an Ehresmann connection are related by:
$$\omk(a,b)(e)=\bigl(\omega(a,b)(e),e\bigr), \quad e\in E. $$

\begin{proposition}\label{prop:vbafibration}
Any \vba $(D;A,E;M)$ determines a Lie algebroid fibration of the form:
\begin{equation}\SelectTips{cm}{}\xymatrix@C=30pt{*+[c]{\pE^*C\ } \ar@{^{(}->}@<-0.5pt>[r]^-{i} \ar[d]&D  \ar@{->>}@<-0.5pt>[r]^-{p} \ar[d]& A\ar@{->}[d] \\
             E \ar@{->}[r]^-{\id_E}&E  \ar@{->}[r]^-{\pE}              & M.}
\end{equation}
\end{proposition}
\begin{proof}
Given an $A$-path $adt:TI\to A$, extended to a time dependant section $a_t\in\Gamma(A)$, the parallel transport $\holK_a:\K|_{E_0}\to\K|_{E_1}$ (if well defined) is given by the trajectories of the time dependent ordinary differential operator defined on sections of $\K$ by $\D_a:\Gamma(\K)\to\Gamma(\K),\ k\to [\hor(a),k]_D$  (see \cite{Br}). Recall that $\D_a$ is a derivation of $\K$, and that such operator is equivalent to a linear vector field on $\K\to E$ whose flow is a Lie algebroid morphism. In the case of a \vba with a linear connection, the symbol of $\D_a$ is itself a linear vector field on $E$, since it is given by $\rho_D(\hor(a))$. In particular, $\rho_D(\hor(a))$ is complete whenever its projection $\rho_A(a)$ to $M$ is complete. Thus any linear connection on a \vba induces a complete Ehresmann connection.
\end{proof}

For a split \vba, parallel transport is always well defined. Clearly, the holonomy in $\K$ induced by the Lie algebroid fibration is
 related to the linear holonomy in $E$ and $C$ in the following way:
\begin{equation}\label{eq:holonomies}
\holK_a(c,e)=\left(\holC_a(c),\holE_a(e)\right) \quad (c,e)\in\K=\pE^* C.
\end{equation}



\subsection{The homotopy sequence of a \vba}

Since \vbas are special cases of Lie algebroid fibrations, one may specialize the long exact sequence of the homotopy groups given by Theorem \ref{thm:longexactfibration}.

\begin{theorem}\label{prop:longexactfibrationvbacase}
Given a \vba  $(D;A,E;M)$, the following assertions hold:
\begin{enumerate}[label=\roman*)]
\item \label{prop:longexactfibrationvbacasei} the homotopy long exact sequence associated with the fibration $\SelectTips{cm}{}\xymatrix@C=15pt{p:D  \ar@{->>}[r]&  A_{}}$ reads
\begin{gather} 
 \pi_k(D)\simeq p_E^*\pi_k(A) \quad\quad\quad\forall k\geq 3,\nonumber\\
     \SelectTips{cm}{}\xymatrix@C=15pt{\pi_2(D)\ \ar@{^{(}->}@<-0.5pt>[r]& \pE^*\pi_{2}(A)\ar@<-0.5pt>[r]^-{\delta_2}& \G(\K) \ar@<-0.5pt>[r]^-{\tilde{\iota}}& \G(A_E) \ar@{->>}@<-0.5pt>[r]^{\pi}& \G(A). }\label{seq:monodromy:inject}
\end{gather}

\item \label{prop:longexactfibrationvbacaseii} the monodromy group $\MON(D,e)$ of $D$ at $e\in E$ fits into an exact sequence of groups
\begin{equation}\label{seq:exact:monodromy}
\SelectTips{cm}{}\xymatrix@C=15pt{\im (\delta_2,e)\  \ar@{^{(}->}@<-0.5pt>[r]& \MON(D,e)   \ar@{->>}@<-0.5pt>[r]^-{p} & \MON(A,m),}
\end{equation}
where $\MON(A,m)$ denotes the monodromy group of $A$ at $m:=p_E(e)$.
\end{enumerate}
\end{theorem}
\begin{proof} It follows from Proposition \ref{prop:int:kernel} that the fundamental groups $\pi_k(\K)$ are trivial for any $k\geq 2$. The first assertion then comes from the exactness of the homotopy sequence of Theorem \ref{thm:longexactfibration}.

In order to prove the second assertion, we denote by $L_D$ the orbit of $D$ through $e$, and by $L_A$ the orbit of $A$ through $m$. As a general fact \cite[Cor.\, 3,12]{BZ}, the monodromy groups fit into exact sequences, as given by the horizontal lines of the following diagram:
\begin{align}\label{diagram:monodromy:chasing}
\begin{gathered}
\SelectTips{cm}{}
\xymatrix{{\pi_2(D,e)}\ \ar@{^{(}->}@<-0.5pt>[r]^{}
                                         \ar@{->}@<-0.5pt>[d]^{}    
                                                          & \pi_2(L_D,e)  \ar@{->>}@<-0.5pt>[r]^-{\tilde{\delta}_2^{CF}}
                                                                          \ar@{->}@<-0.5pt>[d]^{} 
                                                                                                 & \MON(D)_e \ar@{->}@<-0.5pt>[d]^{} \\
                                        \pi_2(A,m)\ \ar@{^{(}->}@<-0.5pt>[r]^{}
                                                          & \pi_2(L_A,m)  \ar@{->>}@<-0.5pt>[r]^-{\delta_2^{CF}}
                                                                                                 & \MON(A)_m,}\end{gathered}                        
\end{align}
where ${\tilde{\delta}_2^{CF}},\ {{\delta}_2^{CF}}$ denote the transgression maps of \cite{CF03} and where the vertical lines are naturally induced by the projection $p:D\to A$.

In this diagram, one shall first observe that the map $\pi_2(L_D,e)\to \pi_2(L_A,m)$ is an isomorphism, which can be argued as follows: 
 by working leaf-wise, one may assume that $A$ is transitive with unique leaf $L_A$, and then choose a splitting $\sigma: TL_A\to A$ of the corresponding Atiyah sequence. By composing $\sigma$ with $\nablaE$ and then restricting to $L_D$, we obtain an Ehresmann connection on the fibration $L_D\to L_A$ whose curvature has values in the distribution $i^*_{L_D}\im\partial\subset T{L_D}$. The leaves of this distribution being affine spaces (this is a consequence of Proposition \ref{prop:int:kernel}) they have trivial topology, from which we deduce that $\pi_2(L_D,e)\simeq \pi_2(L_A,m)$.

  Observe now that the map $\pi_2(D,e)\to \pi_2(A,m)$ is an injection and is such  that $\pi_2(A,m)/\pi_2(D,e)\simeq \im (\delta_2,e)$ (this follows from \eqref{seq:monodromy:inject}). The exact sequence \eqref{seq:exact:monodromy} then easily follows by diagram chasing in \eqref{diagram:monodromy:chasing}.
\end{proof}


\section{Strict $2$-groupoids and $2$-representations}\label{sec:2rep}

In this section we introduce the notion of a representation of a $2$-groupoid, which can be thought of as a categorified version of the usual notion of a representation of a Lie groupoid. The idea is the following: just as Lie groupoids can be represented on a vector bundle, Lie $2$-groupoids are represented in $2$-vector bundles. That is, a $2$-representation is an action of a $2$-groupoid via symmetries of a $2$-vector bundle.

\subsection{Strict $2$-groupoids}

Recall that a strict $2$-category is defined as a category enriched over the category of small categories $\mathbf{\mathrm{\bf Cat}}$.
\begin{definition}
A {\bf strict $2$-groupoid} is a strict $2$-category where all $1$-morphisms and $2$-morphisms have inverses. 
\end{definition}

More explicitly, a strict $2$-groupoid is given by a space of objects $M$, a space of $1$-morphisms $\G^{(1)}$ between objects, and a space of $2$-morphisms $\G^{(2)}$ between $1$-morphisms. The space of $1$-morphisms comes equipped with a strictly associative composition law and strict inverses, so that $1$-morphisms form a groupoid over objects, while the space of $2$-morphisms comes equipped with a horizontal and a vertical composition, so that $\G^{(2)}$ is both a groupoid over $\G^{(1)}$ and over $M$. 

Furthermore, it is required that horizontal and vertical compositions commute, namely:  
$$  (\sigma_4\circV \sigma_3)\circH (\sigma_2\circV \sigma_1)= (\sigma_4\circH \sigma_2)\circV (\sigma_3\circH \sigma_1), $$
for any $\sigma_i, i=1,...,4\in \G^{(2)}$ for which the above makes sense.  We shall denote a strict 2-groupoid $2$-$\G$ and we represent 2-$\G$ diagramatically as follows:

\begin{equation}\label{tgroupoid}
2\text{-}\G: \begin{gathered} \SelectTips{cm}{}\xymatrix@C=5pt{\G^{(2)} \ar@<0.5ex>[rr]^{\sV,\tV}  \ar@<-.5ex>[rr]
                                            \ar@<0.5ex>[dr] \ar@<-.5ex>[dr]_{\sH,\tH} &          & \G^{(1)}   \ar@<0.5ex>[dl]^{\s,\t}
                                                                                                                 \ar@<-.5ex>[dl]\\
                                                                             &\G^{(0)} }\end{gathered}
\end{equation}

Consider now a $2$-groupoid $2$-$\G$ as in \eqref{tgroupoid}. We call \textbf{$1$-truncation} of $2$-$\G$ the orbit space $\G^{(1)}/\mathcal{\G}^{(2)}$, where two $1$-morphisms are identified if there exists a $2$-morphism between them.  The following proposition is straightforward.

\begin{proposition}\label{1truncation}
Let $2$-$\G$ be a $2$-groupoid as in \eqref{tgroupoid}. The $1$-truncation $\G^{(1)}/\mathcal{\G}^{(2)}$ inherits a natural groupoid structure over $M$.
\end{proposition}

\subsection{The Weinstein \texorpdfstring{$2$}{2}-groupoid of a Lie algebroid}\label{sec:Weinsteintwogroupoid}

In this section we are concerned with $2$-groupoids associated to Lie algebroids. More precisely, to any Lie algebroid $p_A:A\to M$, we define its Weinstein $2$-groupoid. The construction is very similar to that of \cite{HKK} for Hausdorff topological spaces. See also \cite{SW, MaPi} for similar constructions in the case $A=TM$.

 Given an $A$-path $adt:TI\to A$ with $a:I\to A$, we denote by $\s(adt)=p_A\circ a(0)$ its source and $\t(adt)=p_A\circ a(1)$ its target.
 An $A$-path with source $x$ and target $y$ will be denoted by $a:x\to y$. The \textbf{inverse} of $a$, denoted by $a^{-1}dt$, is defined to be the $A$-path
$$a^{\mathbf{-1}}(t)dt:=-a(1-t)dt.$$
Every point $x\in M$ determines a trivial $A$-path $\mathbf{1}_x:=0_xdt$, where $0_x\in A_x$. The $A$-path $\mathbf{1}_x$ is called \textbf{unit path} at $x$. Note that at the moment, both notions of units and inverses stand as formal ones. However, they become genuine ones only after moding out by thin homotopies, see below.

\begin{definition}
 An $A$-path $adt:TI\to A$ is {\bf flat at its boundaries} if for any smooth section $\theta\in \Gamma(A^*)$,
 the map $\langle \theta,a\rangle:I\to \mathbb{R}$ vanishes flatly at $t=0$ and $t=1$. That is, $ \langle\theta,a\rangle$ vanishes at $t=0$ and $t=1$ as well as all its higher derivatives.
\end{definition}

\begin{definition}
Given two $A$-paths $adt,bdt:TI\to A$ with $\t(adt)=\s(bdt)$, their {\bf concatenation} $(a \cdot b)dt$ is defined by
\begin{equation*}(b\cdot a)(t)dt =\left\{ \begin{aligned} 
       & {2}a(2t)dt,  & \text{ if } t\in[0,{1}/{2}],\\
       & {2}b(2t-1)dt, & \text{ if } t\in[{1}/{2},1].
        \end{aligned}\right.
\end{equation*}
\end{definition}
Obviously, $(b\cdot a)dt$ is smooth whenever $a$ and $b$ are flat at their boundaries.
In order to concatenate arbitrary $A$-paths, one can replace them by their {\bf reparametrizations}. For that, we consider  a cut-off function $\tau:I\to I$, namely $\tau$ can be taken to be the restriction of any smooth function $f:\mathbb{R}\to [0,1]$ such that $f|_{(-\infty,0)}=0$
and $f|_{(1,\infty)}=0$, so that all the derivatives of $\tau$ vanish at $0$ and $1$.
Then the reparametrization $a^\tau$ is the $A$-path given by $a^\tau(t)dt:= a\circ d\tau:TI \to A$, where $a^\tau(t):=\tau'(t)(a\circ\tau)(t)$.

\begin{definition}
An $\mathbf{A}$\textbf{-homotopy} between two $A$-paths $a_0dt$ and $a_1dt$ is a Lie algebroid morphism $\sig=adt+b ds:TI^2\to A$ such that $a_0dt = adt|_{\{s=0\}}$ and $a_1dt = adt|_{\{s=1\}}$, and satisfying the boundary conditions $b|_{\{t=0,1\}}=0$.
     \end{definition}
 We shall denote by $\sV(\sig)=a_0dt$ and $\tV(\sig)=a_1dt$ the {\bf source} and {\bf target} of $\sig$, and write $\sig:a_0\Rightarrow a_1$.
A simple way to picture an $A$-homotopy is as follows:
\begin{equation*}
\begin{gathered}\SelectTips{cm}{}\xymatrix{x \ar[r]^{a_1 dt}  \ar@{}[dr]|{\sig}             &   y            \\
                                              x\ar@{..>}[u]^{0 ds}\ar[r]_{a_0 dt}  & y\ar@{..>}[u]_{0 ds}}\end{gathered}
																							\quad\quad\begin{gathered}\xymatrix@!@C=50pt{x \ar@/_1,1pc/[r]|{\overset{ }{\, a_0\, }}="d"   \ar@/^1,1pc/[r]|{\, a_1\, }="e" & y \ar@{=>}^-{\sig}"d";"e"-<0pt,4pt>}
\end{gathered}
\end{equation*}

\begin{definition}
 An $A$-homotopy is called {\bf thin} if the induced application $\wedge^2 \sig:\wedge^2 TI^2\to \wedge^2 A$ is trivial.
 Two $A$-paths are said to be {\bf thin homotopic} if there exists a thin homotopy $\sig:a_0\Rightarrow a_1$.
\end{definition}

\begin{definition}
An $A$-homotopy $\sig=adt+bds:TI^2\to A$ is said to be {\bf flat at its boundary} if, for any smooth section $\eta\in \Gamma(A^*)$,
the following conditions holds:
\begin{itemize}
\item[i)] the application $\langle \eta,b\rangle:I\to \mathbb{R}$ vanishes flatly at $\{s=0,1\}$,
\item[ii)] the application $\langle \eta,a\rangle:I\to \mathbb{R}$ vanishes flatly at $\{t=0,1\}$.
\end{itemize}
\end{definition}
Note in particular that in this case, both $A$-paths $\sV(\sigma)$ and $\tV(\sigma)$  are flat at their boundaries.
\begin{definition}
Given two $A$-homotopies $\sig=adt+bds$ and $\sig'=a'dt+b'ds$ with $\tV(\sig)=\sV(\sig')$,
we define their \textbf{vertical concatenation} by
\begin{equation*}\bigl(\sig' \circV \sig \bigr)(t,s) =\left\{ \begin{aligned} 
       & a(t,2s)dt+{2}b(t,2s)ds,  & \text{ if } s\in[0,{1}/{2}],\\
       & a'(t,2s-1)dt+{2} b'(t,2s-1), & \text{ if } s\in[{1}/{2},1],
        \end{aligned}\right.
\end{equation*}
\end{definition}
The vertical concatenation may be pictured in the following way:
$$
\xymatrix@!@C=50pt{ x \ar@/_1,5pc/[r]|{\overset{ }{\, a_0\, }}="a"   \ar[r]|{\overset{ }{\, a_1\, }}="b" \ar@/^1,5pc/[r]|{\  a_2\, }="c" & y \ar@{=>}^-{\sig\ }"a";"b"-<0pt,4pt> \ar@{=>}^-{\sig'}"b";"c"-<0pt,4pt>}
\quad\longmapsto\quad
\xymatrix@!@C=50pt{x \ar@/_1,1pc/[r]|{\overset{ }{\, a_0\, }}="d"   \ar@/^1,1pc/[r]|{\, a_2\, }="e" & y \ar@{=>}^{\sig'\circV\sig}"d";"e"-<0pt,4pt>}$$
The vertical concatenation is smooth whenever $\sigma$ and $\sigma'$ are flat at their boundaries (see \cite[Lem. 3,7]{BZ}). In order to concatenate smoothly two arbitrary $A$-homotopies, 
one has to replace them by their respective reparametrizations, where the \textbf{reparametrization}
of an $A$-homotopy $\sigma$ is given by $\sigma^{\tau}:=\sigma\circ d(\tau\times\tau)$, namely
$\sigma^{\tau}=\tau'(t)a(\tau(t),\tau(s))dt+\tau'(s)b(\tau(t),\tau(s))ds.$

\begin{definition}
The \textbf{vertical inverse} of an $A$-homotopy $\sigma=adt+bds$ is given by
\begin{equation*}
\bigl(\sig^{{\bf -1_V}}\bigr)(t,s)=a(t,s)dt-b(t,1-s)ds.
\end{equation*}
The \textbf{vertical unit} at an $A$-path $adt$ is the $A$-homotopy $\mathbf{1}^V_{adt}:=a(t)dt+0_{\gamma(t)}ds$ (so that $\mathbf{1}^V_{adt}:adt\Rightarrow adt$).
\end{definition}
\begin{definition}
The horizontal source and target maps $\sH,\tH $ of an $A$-homotopy are given by $\sH=\s\circ\sV$ and $\tH=\t\circ\sV$.
Given two $A$-homotopies $\sig=adt+bds$ and $\sig'=a'dt+b'ds$ such that $\tH(\sig)=\sH(\sig')$,
one defines their \textbf{horizontal concatenation} in the following way:
\begin{equation*}\bigl(\sig' \circH\sig \bigr)(t,s) =\left\{ \begin{aligned} 
       & {2}\,a(2t,s)dt+b(2t,s)ds,  & \text{ if } t\in[0,{1}/{2}],\\
       & {2}\,a'(2t-1,s)dt+b'(2t-1,s)ds, & \text{ if } t\in[{1}/{2},1].
        \end{aligned}\right.
\end{equation*}\end{definition}
As easily checked, the horizontal concatenation is smooth whenever $\sigma, \sigma'$ are flat at their boundaries, and may be illustrated as follows:
$$
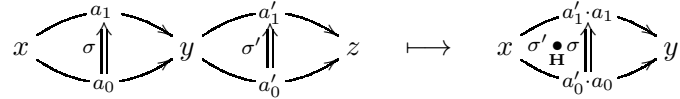
\xymatrix@!@C=50pt{x \ar@/_1.1pc/[r]|{\overset{ }{\, a_0\, }}="a"  \ar@/^1.1pc/[r]|{\, {a_1}_{\,} \, }="b"
                  & y \ar@/_1.1pc/[r]|{\overset{ }{\, a'_0\, }}="c"  \ar@/^1.1pc/[r]|{\, {a'_1} \, }="d" 
							  	  & z \ar@{=>}^-{\sig}"a";"b"-<0pt,4pt>\ar@{=>}^-{\sig'}"c";"d"-<0pt,4pt>}
\quad\longmapsto\quad
\xymatrix@!@C=50pt{x \ar@/_1.1pc/[r]|{{ }{\, a'_0\cdot a_0\, }}="a"  \ar@/^1.1pc/[r]|{\, a'_1\cdot a_1 \, }="c" & y \ar@{=>}^{\sig'\circH \sig}"a";"c"-<0pt,4pt>}$$

\begin{definition}
The \textbf{horizontal inverse} of an $A$-homotopy is defined by:
\begin{equation*}
\bigl(\sig^{{\bf -1_H}}\bigr)(t,s)=-a(1-t,s)dt+b(1-t,s)ds.
\end{equation*}
Given $x\in M$, the \textbf{horizontal unit} is the $A$-homotopy
$\mathbf{1}^{\scriptscriptstyle{H}}_x:=0_xdt+0_xds,$ between the unit path $\mathbf{1}_x$ and itself. 
\end{definition}

\begin{definition}
 We call $\mathbf{3}$\textbf{-homotopy} any Lie algebroid morphism $H=H_1 dt+H_2 ds + H_3 du:TI^3\to A$ satisfying the boundary conditions:
 $H_3 du|_{\{t=0,1\}} =0\, du$, and  $H_3 du|_{\{s=0,1\}} =0\,du$.   
\end{definition}
 In that case, $H$ defines a $3$-homotopy between $\sig_0:=(H_1 dt+ H_2 ds)|_{\{u=0\}}$ and
 $\sig_1:=(H_1 dt+ H_2 ds)|_{\{u=1\}}$. Note that, as a consequence of $H$ being a Lie algebroid morphism,
$\sig_1$ is an $A$-homotopy if, and only if, $\sig_0$ is (see \cite{BZ}). In this work, we shall only consider $3$-homotopies between $A$-homotopies.
Note also that any $A$-homotopy $\sig$ is $3$-homotopic to its reparametrization $\sig^\tau$.

 A $3$-homotopy between two $A$-homotopies $\sig_0$ and $\sig_1$ can be pictured in the following way:
$$  \begin{gathered}\SelectTips{cm}{}
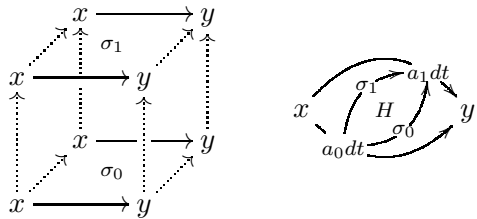
\xymatrix@!0{
                                                                       &  x\ar[rr]\ar@{<..}'[d][dd]\ar@{}[dr]|{\sig_1} &                                                 &  y\ar@{<..}[dd] \\
       x \ar@{..>}[ur]\ar[rr]\ar@{<..}[dd]\ar@{}[dr]|{}\ar@{}[ddrr]|{} &                                               & y \ar@{..>}[ur]\ar@{<..}[dd]\ar@{}[dr]|{}                         \\
                                                                       &  x\ar'[r][rr]\ar@{}[dr]|{\sig_0}               &                                                 &    y            \\
       x  \ar[rr]\ar@{..>}[ur]                                         &                                               & y \ar@{..>}[ur]
}\end{gathered}
\quad\quad\begin{gathered}\xymatrix@!@C=50pt{x \ar@{}[r]|{H}\ar@/_1,5pc/[r]|<<<{\overset{ }{\, a_0dt\, }}="d"   \ar@/^1,6pc/[r]|<<<<<<<<<<<{\, a_1dt\, }="e" & y \ar@/_0,9pc/@{->}|{\sig_0}"d";"e"\ar@/^0,9pc/@{->}|{\sig_1}"d";"e"}
\end{gathered}$$
Where the dotted squares represent trivial morphisms.

\begin{definition}
Given a Lie algebroid $A\to M$, we call {\bf Weinstein $2$-groupoid of $A$}, denoted by  $\tPa(A)$,
the strict $2$-groupoid with $M$ as space of objects, where $1$-morphisms are thin homotopy classes of $A$-paths,
and where $2$-morphisms are given by $3$-homotopy classes of $A$-homotopies, with horizontal and vertical compositions given by the concatenations
along the $t$ and $s$ variables respectively.
\end{definition}

Routine computations show that $\tPa(A)$ is indeed a strict $2$-groupoid, details will be left to the reader. Let us rather emphasize the notations: in the sequel, we shall denote by $P_1(A)$ the space of thin homotopy classes of $A$-paths, and by $P_2(A)$ the space of $3$-homotopy classes of $A$-homotopies, so that $\tPa(A)$ denotes the strict $2$-groupoid structure with $P_1(A)$ as $1$-morphisms, and $P_2(A)$ as $2$-morphisms:

$$\tPa(A): \begin{gathered}\SelectTips{cm}{}\xymatrix@C=5pt{
           P_2(A) \ar@<0.5ex>[rr]^{\sV,\tV} 
                  \ar@<-.5ex>[rr] \ar@<0.5ex>[dr]
                                  \ar@<-.5ex>[dr]_{\sH,\tH} &             & P_1(A)   \ar@<0.5ex>[dl]^{\s,\t}
                                                                                       \ar@<-.5ex>[dl]\\
                                                              &   M
                             }\end{gathered}$$

As explained in Proposition \ref{1truncation}, the $1$-truncation of a $2$-groupoid with objects $M$, inherits the structure of a groupoid over $M$. Applying this construction to the Weinstein $2$-groupoid of a Lie algebroid yields the following result.                                                        
                           
\begin{proposition}\label{prop:weinsteintruncation}                         
Let $A\to M$ be a Lie algebroid. The $1$-truncation $P_1(A)/P_2(A)$ of the Weinstein $2$-groupoid of $A$ coincides with the Weinstein groupoid $\G(A)$ of $A$.
\end{proposition}

\begin{proof} The proof is straightforward and it is a consequence of combining the following facts: any $A$-path is thin homotopic to its reparametrization and thin homotopies are genuine homotopies.

\end{proof}

The Weinstein $2$-groupoid can be used also to recover the fundamental groups $\pi_2(A,m)$ defined in \cite{BZ}, as the isotropies in $P_2(A)$ at identities. Namely, for every $m\in M$ we have

 $$\pi_2(A,m) =\set{\sigma\in P_2(A)\ |\   \sigma:0_mdt\Rightarrow 0_mdt}.$$

\begin{remark}
 In this work, we will not define a differential structure on $\tPa(A)$, hence neither shall we differentiate a representation of $\tPa(A)$ on a $2$-term complex $\E$ to a $2$-term representation of $A$. This was done in the case $A=TM$  by \cite{SW} in a slightly different context (see the precise relation in \cite{CaFl}) where a similar path $2$-groupoid of a smooth manifold is constructed, as well as a smooth structure on it. 

 In \cite{SW}, the smooth structure is defined in terms of  diffeological spaces, and it seems clear that their construction extends to the case of an integrable Lie algebroid. It would be interesting to see if the diffeological structure is also well defined in the non integrable case. 

 Additionally, note that the space of $2$-morphisms defined in this work is slightly smaller than that of \cite{SW, CaFl} since we take quotients by $3$-homotopies. The basic reason is to obtain a $2$-groupoid as small as possible for $2$-term representations. In the terminology of \cite{SW}, we essentially deal here with \emph{flat} $2$-connections (meaning that the curvature $3$-form vanishes \eqref{eq:ruh:omegaclosed}) and hence it is natural to mod out the useless information (see the Lemma \ref{lem:3homotopyindependant}).
\end{remark}

\subsection{The gauge \texorpdfstring{$2$}{2}-groupoid of a \texorpdfstring{$2$}{2}-term complex}

Let $M$ be a smooth manifold. According to Proposition \ref{Dold-Kan}, the category $\tVec_M$ of $2$-vector bundles over $M$ is equivalent to the category $\tTerm_M$ of $2$-term complexes of vector bundles over $M$. In this section we study symmetries of $2$-vector bundles via their description as a $2$-term complex. Just as a vector bundle has a groupoid of symmetries, namely its gauge groupoid, a $2$-vector bundle has a \emph{$2$-groupoid} of symmetries, its gauge $2$-groupoid.

Let $\mathcal{E}=(C\xrightarrow{\d} E)$ be a $2$-term complex of vector bundles over $M$. The \textbf{gauge 2-groupoid} of $\mathcal{E}$ is defined as the strict $2$-groupoid $2\text{-Gau}(\E)$ whose space of objects is $M$, and
\begin{itemize}
\item \textbf{1-morphisms:} a $\mathbf{1}$-morphism $A:x\overset{}{\to} y $ between objects $x,y\in M$ is an invertible chain map. That is, a couple of invertible linear maps $A=(\AC,\AE)$ where $\AC\in\Iso(C_x,C_y)$, $\AE\in\Iso(E_x, E_y)$ and such that 
$\AE\circ\, \partial_x=\d_y\circ \AC$. The composition and inverses of chain maps is
\begin{align*} \bigl(\BC,\BE\bigr)\cdot \bigl(\AC,\AE\bigr)  &=  \bigl(\BC\circ \AC,\BE\circ \AE\bigr) \\
               \bigl(\AC,\AE\bigr)^{-1}                      &=  \bigl((\AC)^{-1}, (\AE)^{-1}\bigr),\end{align*}

\item \textbf{2-morphisms:} for any objects $x,y\in M$ and $1$-morphisms $A_0,A_1:x\to y$, a $\mathbf{2}$-morphism between $A_0$ and $A_1$ is given by a chain homotopy $\phi:A_0{\Rightarrow} A_1$, \emph{i.e.} a linear application $\phi\in \text{ Hom}( E_{x} , C_{y})$ satisfying:
\begin{align*}
\AC_1-\AC_0&= \phi\circ\partial, \\
\AE_1-\AE_0&=\partial\circ \phi,\end{align*}

\item \textbf{vertical composition:} given $2$-morphisms $A_0\overset{\phi}{\Rightarrow} A_1\overset{\psi}{\Rightarrow} A_2$ where $A_0,A_1,A_2$ are $1$-morphisms $x\to y$, the vertical composition between $\phi$ and $\psi$ is the $2$-morphism $\psi\circV\phi:A_0{\Rightarrow} A_2$ defined by:
$$\psi\circV\phi=\psi+\phi,$$

\item \textbf{vertical unit:} for any $1$-morphism $A:x \to y$, the vertical unit at $A$ is the $2$-morphism $1^\textbf{V}_A:A\Rightarrow A$ given by $1^\textbf{V}_A=0 \in \mathrm{Hom}(E_x,C_y)$,

\item \textbf{vertical inverse:} given a $2$-morphism $\phi:A_0\Rightarrow A_1 \in \Hom(E_x,C_y)$, the vertical inverse of $\phi$ is the $2$-morphism $\phi^{-1^\textbf{V}}:A_1\Rightarrow A_0 $ given by $\phi^{-1^\textbf{V}}:=-\phi$
 
\item \textbf{horizontal composition:} given $2$-morphisms $A_0\overset{\phi}{\Rightarrow} A_1$, $A'_0\overset{\phi'}{\Rightarrow} A'_1$ where $A_0, A_1:x\to y$ and $A'_0,\ A'_1:y\to z$ are $1$-morphisms, the horizontal composition
$\phi'\circH\phi:A'_0\circ A_0\Rightarrow A'_1 \circ A_1$ is given by:
$$\phi'\circH \phi =\phi'\circ A_0+A'_1\circ\phi.$$

\item \textbf{horizontal unit:} for any $x\in M$, the horizontal unit $1^\textbf{H}_x$ is the $2$-morphism $1^\textbf{H}:\id\Rightarrow \id$ given by the trivial morphism $1^\textbf{H}:=0\in \Hom(E_x,C_x)$,

\item \textbf{horizontal inverse:} given a $2$-morphism $\phi:A_0\Rightarrow A_1 \in \Hom(E_x,C_y)$, the horizontal inverse of $\phi$ is the $2$-morphism $\phi^{-1^\textbf{H}}:A_0^{-1}\Rightarrow A_1^{-1} $ given by:
$$  \phi^{-1^\textbf{H}}:=-A_1^{-1}\circ\phi\circ A_0^{-1}\in \Hom(E_y,C_x). $$ 
\end{itemize}

It is easy to check that $2\text{-Gau}(\E)$ indeed defines a strict $2$-groupoid. This is left to the reader.

\begin{remark}\label{rem:gauge:auto}
It is easily seen by adapting the results Baez and Crans \cite{BaezCrans6}  that $2\text{-Gau}(\E)$ is isomorphic  to the strict $2$-groupoid $\tGau(\K)$ of linear automorphisms of $\K=\pE^*C$, where $\K$ is seen as a $2$-vector bundle by Proposition  \ref{prop:int:kernel}. Here $\tGau(\K)$ is obtained by taking objects to be points in $M$, $1$-morphisms to be invertible linear functors $\K_x\to \K_y$, and $2$-morphisms to be linear natural transformations between invertible linear functors.

Therefore one can of think of $\tGau(\E)$ as a \emph{frame $2$-groupoid} for the $2$-vector bundle $\K$.  Since our main concern is with $2$-term representations up to homotopy, we will work with $\tGau(\E)$, rather than $\tGau(\K)$, although the latter would be more natural from the point of view of fibrations. 
\end{remark}

\begin{remark}
In the special case when $\d:C\to E$ has constant rank, both spaces of $1$-morphisms and of $2$-morphisms of $\tGau(\E)$ come with an obvious structure of a (finite dimensional) smooth manifold. Furthermore, all the structure maps are easily seen to be smooth. In that case, $\tGau(\E)$ is a Lie $2$-groupoid. This assumption, however, is quite restrictive in view of the applications. For instance, one of the main motivations for introducing representations up to homotopy was to get a model for the adjoint and coadjoint representations of a Lie algebroid $A\to M$. In that case, $\d=\rho_A:A\to TM$ coincides with the anchor map and therefore it has not constant rank unless $A$ is regular.

Since this condition is irrelevant for the algebraic constructions of the next sections to hold, we shall not assume that $\partial:C\to E$ has constant rank.
\end{remark}


\subsection{The holonomy $2$-representation}     

Let $\G\rightrightarrows M$ be a Lie groupoid. A representation of $\G$ consists on a vector bundle $E\to M$ equipped with a linear action of $\G$, i.e. a groupoid morphism $\G\to \text{Gau}(E)$. If one moves to $2$-groupoids, 
the corresponding notion of linear action is that of a $2$-representation.

\begin{definition}\label{def:2rep}
Let 2-$\G$ be a $2$-groupoid as in \eqref{tgroupoid}. A \textbf{$2$-representation} of 2-$\G$ consists of a $2$-term complex $\E=(\partial: C\to E)$, equipped with a strict $2$-functor $\Phi: 2\text{-}\G\to \tGau(\E)$.
\end{definition}

More precisely, if $\G^{(2)}, \G^{(1)}$ and $M$ denote the spaces of $2$-morphisms, $1$-morphisms and objects of 2-$\G$, respectively, then a $2$-representation is given by the following assignment:
\begin{itemize}
\item to any $1$-morphism $g:x\to y$, one associates an invertible chain map $\Phi_g=(A^C_g,A^E_g)$:
$$\xymatrix{C_x \ar@{->}[r]^{\partial_x}\ar[d]_{A^C_g}&E_x  \ar[d]^{A^E_g}\\
             C_y \ar@{->}[r]^{\partial_y}&E_y }$$
in such a way that $\Phi_{gh}=\Phi_g\circ \Phi_h$, for every composable pair of $1$-morphisms $g,h\in \G^{(1)}$,

\item to any $2$-morphism $\sigma\in \G^{(2)}$ with $\s_V(\sigma)=g$, $ \t_V(\sigma)=h$, $\s_H(\sigma)=x$ and $\t_H(\sigma)=y$, one associates a chain homotopy $\Phi_\sigma\in \Hom(E_y,C_x)$ between the chain maps $\Phi_g$ and $\Phi_h$.  The assignment $\sigma \mapsto \Phi_\sigma$ is compatible with both the vertical and horizontal composition, as well as with inverses.         
\end{itemize}




Note that a usual representation $E\to M$ of a Lie groupoid $\G\tto M$ can be thougth of as a $2$-representation of the $2$-groupoid $2$-$\G$ (objects are points of $M$, the space of $1$-morphisms is $\G$ and the space of $2$-morphisms is also $\G$) on the $2$-term complex $0:E\to E$.

A $2$-term representation up to homotopy $(\nablaE,\nablaC,\omega)$ of a Lie algebroid $A$ on a $2$-term complex $\E=(\partial:C\to  E)$ gives rise to a $2$-representation of the Weinstein $2$-groupoid $\tPa(A)$ on $\E$. Indeed, in that situation one defines the corresponding \textbf{holonomy}, denoted by $\hol$, first as an assignment, in the following way:
\begin{itemize}                                                                                           
\item to any $A$ path $adt:TI\to A$ from $x$ to $y$, the holonomy associates the couple: 
$$\hol(a):=\left(\holC_a, \holE_a\right)$$
 where ${\holC_a}:C_{x} \xrightarrow{} C_{y}$ (respectively                                                                                                                          
       ${\holE_a}:E_{x}     \xrightarrow{} E_{y}$) denotes the holonomy of $\nablaC$ (respectively $\nablaE$) along $a$,                                     
\item to any $A$-homotopy $\sig=adt+bds:TI^2\to A$, the holonomy associates the application $\hol(\sig)\in\Hom(E_{x},C_{y})$ defined by:                                                                                                
\begin{equation}\label{eq:holonomy-homotopy}                                                                                                                                          
\hol(\sig)=\int_0^1 \!\!\! \int_0^1 \holC_{a_{1,t}^s}\circ\  \omega(a,b)_{\gamma_{t}^{s}}\circ \holE_{a_{t,0}^s} dt\, ds.                           
\end{equation}
\end{itemize}
Here, the notations in \eqref{eq:holonomy-homotopy} are detailed in  Appendix \ref{sec:thm:proof}.  We can now state one of the central results of this work.
\begin{theorem}\label{thm:2functor}
Let $(\nablaE,\nablaC,\omega)$ be a representation up to homotopy of a Lie algebroid $A$ on a $2$-term complex of vector bundles
$\E=(\partial:C\to E)$. The holonomy defined above descends to a strict $2$-functor $\hol:\tPa(A)\to \tGau(\E)$
covering the identity on $M$.
\end{theorem}

For clarity, the proof of Theorem \ref{thm:2functor} is postponed to Appendix \ref{sec:thm:proof}. 

\begin{definition}\label{def:2holonomy} The $2$-functor $\hol:\tPa(A)\to 2\text{-}\Gau(\E)$ is referred to as the \textbf{holonomy $2$-representation} associated to the $2$-term representation up to homotopy $(\nablaE, \nablaC, \omega)$ of $A$ on $\E=(\d:C\to E)$.
\end{definition}

\begin{remark}
 If $A=TM$, both $E=M\times \R^n$ and $C=M\times \R^k$ are trivial vector bundles and $\partial:C\to E; (x,v)\to (x,f(v))$ where $f:\R^n\to \R^k$ is a linear map, Theorem \ref{thm:2functor} recovers the work of Schreiber and Waldorf \cite{SW} (from which  Theorem \ref{thm:2functor} was directly inspired). An alternative proof should also follow from the integration procedure of Arias Abad and Sch\"{a}tz 
 \cite{CA01} (see also \cite{CaFl} that relates both approaches).

 
  Note also that working directly with the Weinstein $2$-groupoid avoids, after pulling back the structure along a morphism
$TI^2\to A$, the choice of a trivialization in order to express the local connection forms as in \cite{SW2} (although the gluing of $2$-functors is a quite remarkable feature).
\end{remark}


\subsection{Transformation $2$-groupoid associated to a $2$-representation}
Given a $2$-representation $\Phi:2\text{-}\G\to \tGau(\mathcal{E})$ of a $2$-groupoid, there exists a transformation object $2\text{-}\G\ltimes \E$, whose outcome is a $2$-groupoid. Since we are mainly interested in holonomy $2$-representations as in Theorem \ref{thm:2functor}, we shall present this construction only in that case, although it is easily seen to make sense in general.

For that, consider a Lie algebroid $A\to M$ and let $\tPa(A)$ the Weinstein $2$-groupoid of $A$, defined in Section \ref{sec:Weinsteintwogroupoid}. Fix a $2$-term representation up to homotopy $(\nablaE,\nablaC,\omega)$ of a $A$ on a $2$-term complex $\E=(\d:C\to E)$. Let $\hol:\tPa(A)\to 2\text{-}\Gau(\E)$ be the corresponding holonomy $2$-representation given by Theorem \ref{thm:2functor}.

\begin{definition}\label{2semidirect}
The {\bf transformation $2$-groupoid} associated to $\hol:\tPa(A)\to 2\text{-}\Gau(\E)$ is the strict $2$-groupoid $\tPa(A)\ltimes \E$ defined as follows:
\begin{itemize}
\item \textbf{objects:} the space of objects of $\tPa(A)\ltimes \E$ is given by $E$,
\item \textbf{1-morphisms:} the space of $1$-morphisms of $\tPa(A)\ltimes \E$ is:
                  $$P_1(A)\ltimes \E:= \t^*C\oplus \s^* E,$$
                  where $\s,\t:P_1(A)\to M$ are the source and target maps of $1$-morphisms of the Weinstein $2$-groupoid. The source, target, and inverse maps of $1$-morphisms are given as follows:
\begin{align*} 
\ts(c,a,e)&=e,\\ 
\tit(c,a,e)&=\holE_{a}(e)+\d c\\
(c,a,e)^{\bf -1}&=\left(-\holC_{a^{-1}}(c),a^{-1},\holE_a(e)+\d c\right),
\end{align*}
\noindent two $1$-morphisms $(c_1,a_1,e_1)$ and $(c_0,a_0,e_0)$ are composable provided $e_1=\holE_{a_0}(e_0)+\d c_0$ and
their multiplication is given by:
\begin{align*}
(c_1,a_1,e_1)\cdot(c_0,a_0,e_0)=\left( c_1+\holC_{a_1}(c_0),a_1\cdot a_0,e_0\right),
\end{align*}
\item \textbf{2-morphisms:} the space of $2$-morphisms in $\tPa(A)\ltimes \E$ is given by:
$$P_2(A)\ltimes \E:= \tH^*C \oplus \sH^*E, $$
where $\sH,\tH:P_2(A)\to M$ are the horizontal source and target maps of the Weinstein $2$-groupoid,

\item \textbf{vertical structure maps:} the vertical source and target maps $\tsV,\ttV: P_2(A)\ltimes \E\to P_1(A)\ltimes \E$ are given by:
\begin{align*}
\tsV(c,\sig,e)&:=(c,\sV(\sig),e)\\
\ttV(c,\sig,e)&:=(c-\hol(\sig)_e,\tV(\sig),e).
\end{align*}
Two $2$-morphisms $(c_2,\sig_2,e_2)$ and $(c_1,\sig_1,e_1)$ are composable vertically provided
$\sV(\sig_2)=\tV(\sig_1)\in P_1(A)$ and $c_2=c_1+\hol(\sig_1)(e_1) $, and their vertical multiplication is given by:

$$\bigl(c_2,\sig_2,e_2\bigr)\circV\bigl(c_1,\sig_1,e_1\bigr)  =  \bigl(c_1, \sig_2\circV \sig_1, e_1\bigr).$$

\noindent The vertical inverse of a $2$-morphism $(c,\sigma,e)$ is defined by the following formula:

$$\bigl(c,\sig,e\bigr)^{-{\bf 1}_{\bf V}}                           =  \bigl( c-\hol(\sig)_e,\sig^{-{\bf 1}_{\bf V}},e\bigr)$$

\item \textbf{horizontal structure maps:} the horizontal source and target maps are given by:
\begin{align*}\tsH(c,\sig,e)&=e,\\
              \ttH(c,\sig,e)&=\hol_{\tV(\sig)}(e)+\partial c, 
\end{align*}

\noindent The horizontal composition and horizontal inverse are given by:
\begin{align*}
\bigl(c,\sig,e\bigr)\circH(b,\tau,f)  &=  \bigl(c+\holC_{\sV(\sig)}(b), \sig\circH \tau ,f \bigr)\\
\bigl(c,\sig,e\bigr)^{-{\bf 1}_{\bf H}}               &=  \bigl(\holC_{\sV(\sig)}(c),  \sig^{-{\bf 1}_{\bf H}}, e \bigr).  
\end{align*}
\end{itemize}
\end{definition}

The next result is a direct consequence of the definition of $\tPa(A)$.

\begin{theorem}\label{thm:semidirect2groupoid}
Let $A$ be a Lie algebroid over $M$. If $\hol:\tPa(A)\to 2\text{-}\Gau(\E)$ is the holonomy $2$-representation associated to a $2$-term representation up to homotopy of $A$, then $\tPa(A)\ltimes \E$ described above is a strict $2$-groupoid.
\end{theorem}

\begin{remark}\label{vb2groupoid}
 Quite interestingly, one can think of the semi-direct product in Definition \ref{2semidirect} as a (strict) analogue of the construction of a Grothendieck fibration (with cleavage) associated  with a pseudo-functor. In fact, one may see a \vbg as a special kind of Grothendieck fibration, for which the choice of a horizontal lift gives a cleavage (so that the corresponding groupoid $2$-representation is indeed a special kind of a pseudo-functor).

Part of this work was motivated by the fact that this construction may not be the most natural from the point of view of \vbas and \vbgs, essentially because the splittings of a \vba do not integrate canonically to horizontal lifts of the corresponding \vbg (as we shall see in Remark \ref{rem:splitingsdonotintegrate}).

There is yet another feature enjoyed by the semi-direct product $\tPa(A)\ltimes \E$ that also pleads in the favour of our construction, namely the fact that $\tPa(A)\ltimes \E$ is a $2$-\vbg over $\tPa(A)$ in the following sense. One may first notice that both the spaces of $1$-morphisms and of $2$-morphisms in $\tPa(A)\ltimes \E$ are \vbgs over $P_2(A)$ and $P_1(A)$ respectively. Both even come with a left horizontal splitting by construction, for which the curvature $\Omega$ vanishes. Then taking into account the $2$-groupoid structures on  $\tPa(A)$ and $\tPa(A)\ltimes \E$, we obtain a diagram as follows:
$$\begin{gathered}\xymatrix{\tPa(A)\ltimes \E\ar[d]\\ \tPa(A)}\end{gathered}\quad\quad
  \begin{gathered}\xymatrix@C=20pt@R=20pt{
           P_2(A)\ltimes \E \ar@<0.4ex>[rr]^{\tsV,\ttV}
                  \ar@<-.4ex>[rr]  \ar@<.4ex>[ddr]
                                                \ar@<-.4ex>[ddr]^<<<<<<<{\tsH,\ttH}
                                                 \ar[d] &             & P_1(A)\ltimes \E   \ar@<0.4ex>[ddl]_<<<<<<<{\ts,\tit}
                                                                                       \ar@<-.4ex>[ddl] \ar[d]\\
           P_2(A) \ar@<0.4ex>[rr]|(.24)\hole^-{\sV,\tV}|(.75)\hole
                  \ar@<-.4ex>[rr]|(.268)\hole^-{}|(.73)\hole \ar@<0.4ex>[ddr]
                                             \ar@<-.4ex>[ddr]_-{\sH,\tH} &             & P_1(A)\ar@<0.4ex>[ddl]^-{\s,\t}
                                                                                       \ar@<-.4ex>[ddl]                                                                                                     \\
                                                              &   E  \ar[d]\\
                                                              &   M
                             }\end{gathered}$$

\noindent where the upper and lower triangles are strict $2$-groupoids, all vertical maps are vector bundles, the three faces are (canonically split, flat) \vbgs, and all structures are compatible with each other in the obvious sense.
\end{remark}




\section{Application to the integration of \texorpdfstring{\vbas}{VB-algebroids}}

\subsection{\vbgs as $1$-truncations}

Consider a Lie algebroid $A\to M$ together with a $2$-term representation up to homotopy $(\nablaE,\nablaC,\omega)$ on a complex $\E$, with holonomy $2$-representation $\hol:\tPa(A)\to 2\text{-}\Gau(\E)$. Let $D$ be the \vba over $A$ associated to such $2$-term representation up to homotopy.

The following result shows how to pass from the transformation $2$-groupoid $\tPa(A)\ltimes \E$ to the Weinstein groupoid of $D$.

\begin{theorem}\label{thm:integrationmodhomotopies}
The Weinstein groupoid $\G(D)$ of $D$ identifies with the $1$-truncation of the transformation $2$-groupoid $\tPa(A)\ltimes \E$ associated to the holonomy $2$-representation.
 
 \end{theorem}
 
More precisely, the groupoid structure on $P_1(A)\ltimes \E\tto E$ descends to an isomorphism
$$\G(D)=\ \quotient{(P_1(A)\ltimes \E)}{\sim}$$
where the groupoid structure on the right hand side is obtained by 1-truncation, as explained in Proposition \ref{1truncation}. The above Theorem says that this quotient can be identified, as a groupoid, with the \emph{topological} Weinstein groupoid $\G(D)$ of the total space $D$ of the corrresponding \vba. Note that this identification is free of choice once the $2$-representation is fixed.

\begin{proof}[of the Theorem \ref{thm:integrationmodhomotopies}] The proof relies on a construction from \cite[Sec. 4]{Br} of which we will only outline the main ideas. Note however that we use here slightly different conventions, this in order to obtain the holonomy as a contravariant $2$-functor.

As detailed in \cite{Br}, given an arbitrary Lie algebroid fibration $D\to A$, the choice of an Ehresman connection $D=\K\oplus \Hor$ induces an identification  of the form
\begin{equation}P_1(D)\simeq P_1(\K)\rtimes P_1(A),\label{eq:splitpaths}
\end{equation}
where  $P_1(\K)\rtimes P_1(A)$ is a short notation for the fibered product
$\{(a_{\K},a)\in P_1(\K)\times P_1(A) : p_E(\t(a_\K))=\s(a) \}$.
By using the identification \eqref{eq:splitpaths} it is possible to describe both the concatenation and the homotopies of $D$-paths directly in $P_1(\K)\rtimes P_1(A)$  (see \cite[Prop.\,4,1 and 4.3]{Br}). With our conventions, the concatenation of $D$-paths reads:
\begin{align}\label{eq:split:concatenation}(a_\K,a)\cdot(b_\K,b)=\bigl(a_\K\cdot (b_\K\circ \holK_a)\,, a\cdot b\bigr).\end{align}
 Furthermore, the following facts emerge from the description of $D$-homotopies in this split form:
\begin{itemize}
\item Firstly, the equivalence relation by $D$-homotopies factors through $\G(\K)\rtimes P_1(A)$ as follows:
 $$\SelectTips{cm}{}\xymatrix@C=-7pt{*+[c]{P_1(\K)\rtimes P_1(A)}\ar@{->>}[rr]\ar@{->>}[dr]& &*+[c]{\G(\K)\rtimes P_1(A)}\ar@{-->>}@<1ex>[dl]\\
                                                                                                                  &*+[c]{ \G(D) }}$$                                                                                                                  
where $\G(\K)\rtimes P_1(A)$ denotes the fibred product $\{(v,a)\in \G(\K)\times P_1(A) \ :\ p_E(\t(v))=\s(a) \}$, and the map $P_1(\K)\rtimes P_1(A) \to \G(\K)\rtimes P_1(A)$ is the identity on the $P_1(A)$ factor, and takes a $\K$-path $a_\K$ to its $\K$-homotopy class $v:=[a_\K]$ of $\K$-path. 

\item Secondly, the factored map $\G(\K)\rtimes P_1(A)\to \G(D)$ is the quotient by the equivalence relation 
$(v_0,a_0) \sim (v_1,a_1)$ iff there exists an $A$-homotopy $\sigma:a_0\Rightarrow a_1$ such that $g_1^{-1}\cdot v_0=\partial_{ext}(\sigma, \s(v_0) )$. Here $\partial_{ext}(\sigma,\s(v_0)) \in \G(\K)$ is represented by a $\K$-path obtained by solving a diferential equation (see below). 

\end{itemize}
 In the case of a fibration induced by a \vba, we know from Proposition \ref{prop:int:kernel} that $\G(\K)=E\oplus C$ so that two elements $(v_0=(e'_0,c_0),a_0)$ and $ (v_1=(e'_1,c_1),a_1)$ in $\G(\K)\rtimes P_1(A)$ are equivalent iff $(c_1-c_0,e_0')=\partial_{ext}(\sigma, e'_0)$ for some $A$-homotopy $\sigma:a_0\Rightarrow a_1$, and $e'_1=\t(\partial_{ext}(\sigma, e'_0))$. Furthermore, by combining the definition of $\partial_{ext}$ in \cite[Prop.\,4.3]{Br} with the Remark \ref{rem:int:kpaths}, we get: $$\partial_{ext}(\sigma, e')=\left(\hol(\sigma)(e), e'\right)\in \G(\K)=\pE^* C, \text{ where } e:=\holE_{a_0^{-1}}(e').$$
 
 Finally, rather than $\G(\K)\rtimes P_1(A)$, one shall work with $P_1(A)\ltimes \E $. Note that both spaces are in bijection by the map $(g=(e',c),a) \mapsto (c,a,e=\hol_a^{-1}(e'))$ so we only have to transport the equivalence relation from 
$\G(\K)\rtimes P_1(A)$ to $P_1(A)\ltimes \E$. We obtain:
 $$(c_0,a_0,e_0)\sim (c_1,a_1,e_1)\iff
 \begin{cases}
 \exists\, (\sig:a_0\Rightarrow a_1)\in P_2(A) : c_1-c_0=\hol(\sig)(e_0),\\ \text{ and } e_0=e_1.
  \end{cases}$$
Therefore, the result follows directly from the construction of the semi-direct $2$-groupoid  $\tPa(A)\ltimes \E$, since there exists a $2$-morphism in $\tPa(A)\ltimes \E$ between $(c_0,a_0,e_0)$ and $(c_1,a_1,e_1)$ exactly if the above condition holds. 

The fact that the groupoid structure also descends is a consequence of the construction of $\tPa(A)\ltimes \E$ as well. Essentially, we transport the concatenation from $P_1(D)$ to $P_1(\K)\rtimes P_1(A)$ obtaining \eqref{eq:split:concatenation} then mod out in the first factor to $\G(\K)\rtimes P_1(A)$ and finally, transport it to $P_1(A)\ltimes \E$, where we recover the composition law of $1$-morphisms. 
\end{proof}

As we will see in Subsections \ref{sec:int:t1} and \ref{sec:int:t0}, Theorem \ref{thm:integrationmodhomotopies} can be used in order to integrate \vbas in a rather explicit manner. For now on, let us point out the following degenerate cases, that justify our statement in the introduction that both the transformation and the semi-direct Lie algebroids should be treated as similar constructions.
\begin{example}
Assume that $C$ is the zero vector bundle over $M$, then necessarily $\omega=0$, and $\nablaE$ is a flat $A$-connection. In that case, we recover the Weinstein groupoid of the transformation algebroid
$A\ltimes E\to E$.
\end{example}
\begin{example}
Assume that $E$ is the zero vector bundle over $M$, then necessarily $\omega=0$, and $\nablaC$ is a flat $A$-connection. In that case, we recover the Weinstein groupoid of the semi-direct product Lie algebroid $A\ltimes C\to M$.
\end{example}

\subsection{Integral criterium for integrability as the image of a transgression map}    

As already mentioned, not every $2$-term representation up to homotopy of an integrable Lie algebroid $A$
comes from a $2$-term representation up to homotopy of the Weinstein groupoid $\G(A)$. This problem was addressed as an integrability problem for the corresponding \vba in \cite{BCO}, where the following result was proved.

\begin{theorem}[\cite{BCO}]\label{thm-integrability:integral}
Let $(\nablaE,\nablaC,\omega)$ be a representation up to homotopy of $A$ on a complex $\E$, and let $D$ denote the corresponding \vba. Then $D$ is integrable if, and only if, $A$ is integrable and for any $[\sigma]\in \pi_2(A)$, $\sigma=adt+bds$, the periods of $\omega$ along $\sigma$ vanish:                       %
$$\int_0^1\!\!\!\int_0^1 \holC_{a_{1,t}^s}\circ\ \omega(a,b)_{\gamma_t^s}\circ \holE_{a^s_{t,0}}(x)dt ds=0.$$                       %
\end{theorem}

Here, we recover the formula for the holonomy applied to $A$-spheres. In fact, the holonomy along $A$-spheres has a nice interpretation in terms of Lie algebroid fibrations: it coincides with the trangression map of Theorem \ref{thm:longexactfibration}.

\begin{proposition}\label{prop:transgression}
Let $(D;A,E;M)$ be a \vba. The transgression map: $$\delta_2:\pE^*\pi_2(A)\longrightarrow \G(\K)\simeq C\oplus E$$ associated with the underlying fibration $\SelectTips{cm}{}\xymatrix@C=15pt{p:D  \ar@{->>}[r]& A_{}}$
coincides with the holonomy along $A$-spheres \eqref{eq:holonomy-homotopy} induced by any splitting of $D$.
 Namely, for any $A$-sphere $\sigma\in\pi_2(A,m)$ based at a point $m\in M$, seen as an element of $P_2(A)$, we have: 
$$\delta_2[\sigma]_e=(\hol(\sigma)(e),e) \quad \forall e\in p_E^{-1}(m)\subset E.$$
\end{proposition}
\begin{proof} We shall use consistently the same notations as in the proof the Theorem \ref{thm:integrationmodhomotopies}.   By construction \cite{Br,BZ}, the transgression map $\delta_2$ coincides with the restriction of $\partial_{ext}$ to $p_E^*\pi_2(A)$. Furhermore, for any $A$-sphere $\sigma$, we have $\partial_{ext}(\sigma,e)=(\hol(\sigma)(e),e)$ since $\sigma:a_0\Rightarrow a_1$ where $a_0$ and $a_1$ are trivial $A$-paths, therefore $e'$ and  $e=\holE_{a_0^{-1}}(e')$ do coincide. 
\end{proof}

\begin{proof}[of the Theorem \ref{thm-integrability:integral}] According to Theorem \ref{prop:longexactfibrationvbacase}, the monodromy groups fit into an exact sequence:
\begin{align}\label{exact:proof:int}
\SelectTips{cm}{}\xymatrix@C=15pt{\im (\delta_2,e)\  \ar@{^{(}->}@<-0.5pt>[r]& \MON(D,e)   \ar@{->>}@<-0.5pt>[r] & \MON(A,m)},
\end{align}
where $\delta_2(\sigma,e)=(\hol(\sigma)(e),e)$ by Proposition \ref{prop:transgression}. Assume that there exists an $A$-sphere $\sig_0$ with $\hol(\sig_0)\neq 0$, and let $e_0\in E$ such that $\hol(\sig_0)(e_0)\neq 0$. Then we can define a sequence of non trivial elements in the monodromy groups as follows: 
$$\xi_n:=\bigl(\,\frac{e_0}{n}, \hol(\sig_0)\bigl(\frac{e_0}{n}\bigr)\,\big)\in
\im (\delta_2,\frac{e_0}{n})\subset \MON(D,\frac{e_0}{n}) \quad\quad  (n\in \mathbb{N}).$$
By the linearity of the application $\hol(\sig_0)$, we see that $(\xi_n)$ defines a sequence of non trivial elements of the monodromy groups of $D$, that converges to a trivial element, namely $0_M^A(m_0)$ where $m_0:=p_E(e_0)$. As a consequence, $D$ can not be integrable.
 Furthermore, if  $D$ is integrable, then $A$ is also integrable because the zero section $0_A^D:A\to D$ defines a Lie algebroid morphism by the axioms of a \vba. 

Reciprocally, if $\hol(\sigma)=0$ for any $A$-sphere $\sig$, then  $\im(\delta_2,e))$ is trivial for any $e\in E$, and it is clear from \eqref{exact:proof:int} that $D$ is integrable
provided $A$ is. 


\end{proof}

\subsection{Integration of type $1$ \vbas}\label{sec:int:t1}

First we recall the definition of a \vba of type $1$ from \cite{GM08}.
\begin{definition}
A \vba $(D;A,E;M)$ is said to be of \textbf{type $1$} if the core anchor $\partial:C\to E$ is a vector bundle isomorphism.	
\end{definition}

As explained in \cite[Sec.\,6]{GM08},
given a \vba $(D;A,E;M)$ of type $1$, one may use $\partial:E\to C$ in order to identify $E$ with $C$. 
By doing so, it follows from \eqref{def:ruh1} that given a splitting of $D$, the induced connections $\nabla^{\scriptscriptstyle{E}}$ and $\nabla^{\scriptscriptstyle{C}}$ identify with each other, while $\omega$ identifies with the curvature $\omegaE$ of $\nablaE$. Consequently, any \vba  of type $1$ identifies with a pull-back Lie algebroid, $D\simeq p_{E }^!A$. As a consequence, $D$ is integrable if, and only if, $A$ is integrable, and hence $\G(D)$ can be obtained as a pull-back groupoid. We explain now how this integration can be recovered from Theorem \ref{thm:integrationmodhomotopies}. For that, consider the $2$-term representation up to homotopy $(\nabla^E,\nabla^E,\omega_E)$ associated to a \vba of type $1$ and let $\tPa(A)\ltimes \E$ be the transformation $2$-groupoid defined by the corresponding holonomy $2$-representation.

\begin{proposition}\label{prop:integrationtype1}

The $1$-truncation of $\tPa(A)\ltimes \E$ identifies with the pull-back groupoid $ p_{E}^!\G(A)$.

\end{proposition}

\begin{proof} First one may observe using sucessively the identification $E\simeq C$ described above together with Lemma \ref{lem:curvature:global}, that the holonomy $2$-functor is given by:
\begin{align}
\hol(a)&=(\holE_a,\holE_a), \quad\quad\quad\ \ \   a\in P_1(A).\\
\label{eq:type1holonomy}
\hol(\sig)&= \holE_{\t_V(\sig)}-\holE_{\sV(\sig)}, \quad \sig\in P_2(A).
\end{align}
In particular, since an $A$-sphere is an $A$-homotopy between trivial paths, the holonomy along any $\sigma\in\pi_2(A)$ is trivial. It follows that the transgression map of Proposition \ref{prop:transgression} vanishes, and we recover the fact that
 $D$ is integrable as a consequence of Theorem \ref{thm-integrability:integral}.
 
Finally, the fact that $D$ integrates to a pull-back groupoid can be obtained by Theorem \ref{thm:integrationmodhomotopies} in the following way.
By the definitions, two $1$-morphisms $(c_0,a_0,x_0)$ and $(c_1,a_1,x_1)$ in $P_1(A)\ltimes \E$ are joined by a $2$-morphism if, and only if, 
the following conditions hold: $x_0=x_1$ and there exists an $A$-homotopy $\sig: a_0\Rightarrow a_1$ such that $c_0=c_1+\hol(\sig)(x_0)$.

Taking the identification $E\simeq C$ and formula \eqref{eq:type1holonomy} into account,
it follows that the quotient space of $1$-morphisms by $2$-morphisms identifies with $ p_{E}^!\G(A)$ by the map
\begin{align*}\quotient{P_1(A)\ltimes \E}{\sim}\ &\longrightarrow \quad p_ E^!\G(A)\\
[c,a,e]\quad\ &\longmapsto (c+\holE_a(e),[a],e).
\end{align*}
Then it is easily checked that the groupoid structure on $P_1(A)\ltimes \E \tto E$ descends to the pull-back groupoid structure on $ p_{E}^!\G(A)$.

\end{proof}

In order to illustrate how the integration of a Lie algebroid $2$-term representation up to homotopy by a $2$-functor differs from a $2$-term representation up to homotopy of the corresponding Weinstein groupoid, which may have been expected as the natural integrating structure, let us take a look at a example proposed in \cite[ex. 2.6]{GM10}.
\begin{example}\label{ex:type1}
Consider the following  \vba:
\begin{align*}\SelectTips{cm}{}
\xymatrix{ TTS^2 \ar[d]_{p_{TTS^ 2}}\ar[r]^-{d{p_{TS^ 2}}} &TS^2\ar[d]^{p_{TS^ 2}}\\
           T S^2 \ar[r]^-{{p_{TS^ 2}}}& S^2}
\end{align*}
where ${p_{TS^ 2}}:TS^ 2\to S^ 2$ and ${p_{TTS^ 2}}:TTS^ 2\to TS^ 2$ denote the natural projection, and $d{p_{TS^ 2}}$ the differential of ${p_{TS^ 2}}$. In that case, we have $\E=(\id: TS^2\to TS^2)$, therefore the \vba $D$ is of type $1$ and  integrates to a pull back groupoid as follows:
\begin{align*}
\SelectTips{cm}{}
\xymatrix{ TS^2\times TS^2 \ar@<0.5ex>[d]\ar@<-.5ex>[d]\ar[r]^{} &S^2\times S^2\ar@<0.5ex>[d]\ar@<-.5ex>[d]\\
           S^2\times S^2 \ar[r]^{}&  S^2,}
\end{align*}
where both $TS^2\times TS^2\tto TS^2$ and $S^2\times S^2 \to S^2$ are pair groupoids.

A splitting of $D$ gives rise to a linear connection $\nablaE$ on $E$ with curvature $\omega_E$. The corresponding $2$-term representation up to homotopy is given by $(\nablaE,\nablaE,\omegaE)$. As explained above, the integrating $2$-functor is simply given by $\hol(\dot{\gamma})=(\holE_\gamma,\holE_\gamma)$ and $\hol(\sigma)=\holE_{\gamma_0}-\holE_{\gamma_1}.$
\end{example}

\begin{remark}\label{rem:Rajehorizontalliftexample}
Comparing Example \ref{ex:type1} with \cite[Ex. 2.6]{GM10}, one may notice that the holonomy $2$-functor is obtained by using only the infinitesimal data $(\nablaE,\nablaE,\omegaE)$, while the construction of an explicit $2$-term representation up to homotopy of $S^2\times S^2\tto S^2$ requires further choices (a Riemannian metric for instance, as  in \cite{GM10}). 

In fact, although the choice of a right-horizontal lift can be differentiated to a $2$-term representation up to homotopy of the underlying Lie algebroid, the converse is \emph{not} true. Namely, the choice of a splitting of a \vba does not integrate to a right-horizontal lift, at least not without involving further choices. It follows that given a $2$-term representation up to homotopy of a Lie algebroid $A$, whose corresponding \vba is assumed to be integrable, we still \emph{do not} obtain a representation up to homotopy of $\G(A)$ in a canonical way.
\end{remark}

\begin{remark}\label{rem:splitingsdonotintegrate}
There is another point that we believe is worth clarifying, which is the following. A $2$-term representation up to homotopy of a Lie groupoid, say $\G(A)$, involves maps $\DeltaE:\G(A)\to \Hom(E,E)$ and $\DeltaC:\G(A)\to \Hom(C,C)$, usually thought of as the holonomy in $\E$ along elements of $\G(A)$.
 
  As emphasized in \cite[Ex. 2.6]{GM10}, in order to cover  a general enough notion of  $2$-term representation up to homotopy for Lie groupoids, one needs to allow $\DeltaE$ and $\DeltaC$ to take values in \emph{non-invertible} homomorphisms.

 This may be quite confusing in view of Theorem \ref{thm:2functor}, since not only do $\holE:P_1(A)\to \Hom(E,E)$ and $\holC:P_1(A)\to \Hom(C,C)$ \emph{not} descend to maps $\G(A)\to \Hom(E,E)$ and $\G(A)\to \Hom(C,C)$, but also do $\holE_a$ and $\holC_a$ \emph{always} define  invertible maps.
 
  Again, the apparent contradiction comes from the fact that the choice of a splitting of $D$ is not enough to induce a right-horizontal lift of $\G(D)$. As a consequence, the corresponding quasi-action $(\DeltaE, \DeltaC)$ is not entirely determined by the infinitesimal data $(\nablaE,\nablaC)$. It is not hard to see however that in general, $(\holE_a,\holC_a)$ and $(\DeltaE_{[a]},\DeltaC_{[a]})$ coincide up to a chain homotopy.
 \end{remark}

\subsection{Integration of type $0$ \vbas}\label{sec:int:t0}

We now explain how to integrate a \vba $(D;A,E;M)$ of type $0$. 
\begin{definition}[\cite{GM08}]
A \vba $(D;A,E;M)$ is said to be of \textbf{type $0$} if the core anchor $\partial:C\to E$ vanishes.	
\end{definition}

Given a splitting of a  \vba of type $0$, it follows from the axioms that the associated $A$-connections $\nablaE,\nablaC$ have vanishing curvatures. Hence both $\nablaE$ and $\nablaC$ are representations of $A$. Furthermore, $\omega$ defines a $2$-cocycle with values in the induced representation of $A$ on $\Hom(E,C)$. Note that $\nablaE,\nablaC$ are canonical in the sense that they are independent of the choice of a splitting of $D$, while 
only the induced class $[\omega]\in H^ 2(A,\Hom(E,C))$ depends on this choice. See \cite{GM08} for more details.

Since both $E$ and $C$ are honest representations, it makes sense to integrate $\omega$ along any $A$-sphere $\sig$ by usual integrals. Note that the resulting integral, called \textbf{period along $\sig$}, only depends on the cohomology class $[\omega]\in H^ 2(A,\Hom(E,C))$ and the homotopy class of $\sig$ in $\pi_2(A)$.  For the sake of simplicity, we shall assume that $E$ and $C$ are trivial vector bundles $E=M\times E_0$ and $C=M\times C_0$ on which $A$ acts trivially. Then for any $A$-path, the holonomy in $E$ and $C$ is the identity. That is 
\begin{align*}
\hol(a)\simeq(\id_{E_0},\id_{C_0}).
\end{align*}
This simplifies the formula \eqref{eq:holonomy-homotopy} for the holonomy along $2$-morphisms, justifying the notation:
 $$\hol(\sig)=\int\!\!\!\! \int_{\sig} \omega$$
The integration of a type $0$ \vba can then be summarized as follows.
\begin{proposition}\label{prop:integrationtype0}
Let $(D;A,E;M)$  be a \vba of type $0$, where both $E$ and $C$ are trivial representations, $E=M\times E_0$ and $C=M\times C_0$, of $A$. Then the following assertions hold:
\begin{itemize}
\item[i)] $D$ is an integrable Lie algebroid if, and only if, $A$ is integrable and for any $[\sig] \in \pi_2(A)$, the periods of $\omega$ along $[\sig]$ vanish.
\item[ii)] the (possibly topological) Weinstein groupoid $\G(D)$ identifies with the quotient of $C_0\times P_1(A) \times E_0$ by the following equivalence relation:
$$
 (c_0,a_0,e)\sim(c_1,a_1,e) \iff \left\{ \begin{gathered} \text{there exists an $A$-homotopy $\sig:a_0\Rightarrow a_1$ such that:}\\
 c_1-c_0=\left(\int\!\!\!\!\!\,\int_{\sig} \omega\right)(e)\end{gathered}\right.
$$
\end{itemize}
\end{proposition}
\begin{proof}
Part $i)$ follows from Theorem \ref{thm-integrability:integral} while the explicit description of the Weinstein groupoid $\G(D)$ is a direct application of Theorem \ref{thm:integrationmodhomotopies}.
\end{proof}

\begin{remark}
Notice the analogy with the construction of the groupoid integrating a prequantization Lie algebroid given by Crainic  (see \cite[Remark 3.3]{Cr-preq}, also \cite{CC}). In fact, for any \vba $(D;A,E;M)$ of type $0$, $D$ can be obtained as a \emph{central} extension in the following way. 

First we denote by $A\ltimes E\to E$ the transformation Lie algebroid. As a general fact \cite[Ex. 2.15]{Br}, $A\ltimes E$ fits into a fibration $\SelectTips{cm}{}\xymatrix@C=15pt{A\ltimes E  \ar@{->>}[r]& A}$ whose kernel is the trivial Lie algebroid over $E$. Next, observe that there is an obvious representation of $A\ltimes E\to E$ on $\pE^* C\to E$ induced by that of $A$ on $C$ by pull-back.  Furthermore, $\omega\in \Omega^2(A\ltimes E,C)$ can be seen as a $2$-cocycle on $A\ltimes E$ with values in $\pE^*C$ by using the obvious inclusion
 $\Omega^2(A,\Hom(E,C))\subset\Omega^2(A\ltimes E,\pE^*C).$
 
Then it is easily seen from the brackets that $D$ identifies with a central extension $D \simeq (A\ltimes E)\ltimes_{\omega} \pE^*C$ with twisting cocycle $\omega$, so that  $D$ fits into the following extension:
\begin{equation}\label{diag:alternative:extension:type0}
\vcenter{\hbox{\SelectTips{cm}{}\xymatrix{{\pE^*C}_{}\ \ar@{^{(}->}@<-0.25pt>[r]^-{i} \ar[d]&D \ar@{->>}@<-0.25pt>[r]^-{p} \ar[d]& A\ltimes E\ar@{->}[d] \\
             E \ar@{->}[r]^{\id_E}&E  \ar@{->}[r]^{\id_E}              & E.}}}
\end{equation}
Note that in this exact sequence, all Lie algebroids are over the same base.
\end{remark}

\begin{example}
There is a $2$-representation up to homotopy one can associate to any finite dimensional Lie algebra $\g$, as was proposed by Sheng-Zhu \cite{SC} in relation with string $2$-algebras. The construction goes as follows.

 Consider the two term complex with trivial boundary $\E:=(C:=\mathbb{R}\xrightarrow{0} E:=\g^*)$ and the $2$-term representation up to homotopy of $\g$ on $\E$ where:
\begin{itemize}
\item $\nablaE:=\ad^*:\g\to \End(\g^*)$ is the coadjoint representation 
\item $\nablaC:=0:\g\to \End(\mathbb{R})$ is the trivial representation,
\item $\omega:=[\ ,\ ]_\g$ is given the Lie algebra bracket on $\g$, seen as an element in $\wedge^2\g^*\otimes\Hom(\g^*,\mathbb{R})$.
\end{itemize}
The associated \vba is given by $D=\g\times\g^*\times \mathbb{R}$, and fits into a Lie algebroid fibration:
$$\SelectTips{cm}{}\xymatrix{\mathbb{R}\times\g^*\,\ar@{^{(}->}@<-0.25ex>[r]\ar[d] &\g\times\g^*\times \mathbb{R}\ar@{->>}@<-0.25ex>[r]\ar[d]& \g\ar[d] \\
                     \g^*       \ar@<-0.25ex>[r]               &\g^*                         \ar@<-0.25ex>[r]            &  \{*\}.}$$
Here, $D$ is indeed of type $0$ so the kernel $\K=\g^*\times\mathbb{R}$ of the fibration is simply a bundle of abelian Lie algebras over $\g^*$.

 The central extension  of Remark \eqref{diag:alternative:extension:type0} takes a particularly interesting form in this example, as we now explain.
 In order to put things into context, recall that given a Poisson manifold $(M,\pi)$ one usually considers the Lie algebroid structure on $T^*\!M$. However, one might see $(M,\pi)$ as a Jacobi manifold \cite{CC} as well. This amounts to see $\pi$ as a $2$-cocycle on $T^*\!M$ with values in the trivial representation $M\times\mathbb{R}$. Then the Lie algebroid structure of $(M,\pi)$, seen as a Jacobi manifold, coincides with the corresponding extension class. More explicitly, it is defined on $T^*\!M\times \mathbb{R}$ with anchor $\rho(\alpha,f)=\pi^\sharp(\alpha)$ and bracket:
\begin{equation*}
\bigl[(\alpha,f),(\beta,g)\bigr]_{T^*\!M\oplus\mathbb{R}}:=\bigl([\alpha,\beta]_{T^*\!M},\Lie_{\pi^\sharp\alpha}g-\Lie_{\pi^\sharp\beta}f+\pi(\alpha,\beta)\bigr),
\end{equation*}
where $[\ ,\ ]_{T^*\!M}$ denotes the standard Lie algebroid bracket on $T^*\!M$ induced by $\pi$. We obtain this way an extension of the form $\SelectTips{cm}{}\xymatrix{ M\times \mathbb{R}\ \ar@{^{(}->}@<-0.20ex>[r] &T^*\! M\times\mathbb{R}  \ar@{->>}@<-0.20ex>[r]& T^*\! M}$, where all the Lie algebroids are over the same base $M$. In the case of a linear Poisson structure $M=\g^*$, one easily checks from the definitions that the Lie algebroid structure on $T^*\!M\times \mathbb{R}$ coincides with $D=\g\times\g^*\times\mathbb{R}$, and we recover the diagram \eqref{diag:alternative:extension:type0} as follows:
$$\SelectTips{cm}{}\xymatrix{\mathbb{R}\times\g^*\,\ar@{^{(}->}@<-0.25ex>[r]\ar[d] &\g\times\g^*\times\mathbb{R} \ar@{->>}@<-0.25ex>[r]\ar[d]& \g\times \g^*\ar[d] \\
                      \g^*       \ar[r]               &\g^*                         \ar[r]            &  \g^*    ,}$$
Finally, note that by applying Proposition \ref{prop:integrationtype0}  to $D$, we recover a result by Cranic-Zhu \cite[Thm. 4, Ex. 4]{CC} describing the integration of $\g^*$ as a Jacobi manifold.
\end{example}

\appendix

\section{Proof of the Theorem \ref{thm:2functor}}\label{sec:thm:proof}
The proof will be divided in several lemmas. 

\begin{lemma}
\label{lem:chain:hom}                                       
For any $A$-path $a:TI\to A$, $\hol(a)=(\hol^E_a,\hol^C_a)$ is a chain homotopy.
\end{lemma}
\begin{proof}
We have to show that $\holE_a\circ\,\d=\d\circ\holC_a$, 
which obviously follows by integrating the relation $\nablaE\circ \d=\d\circ\nablaC$.
\end{proof}

We need to introduce some notations. Given an $A$-homotopy $adt+bds:TI^2\to A$, we shall denote $\gamma_t^s$ rather than $\gamma(t,s)$ the base map. Similarly, for any $t,t'\in[0,1]$, we will denote $a^s_{t',t}:T[t,t']\to A$ the $A$-path $a|_{[t',t]\times\{s\}}dt$, while
\begin{align*}\holE_{a^s_{t'\!\!,t}}&:E_{\gamma^s_t}\to E_{\gamma^s_{t'}}, \\
                \holC_{a^s_{t'\!\!,t}}&:C_{\gamma^s_t}\to C_{\gamma^s_{t'}}
\end{align*}
will denote the corresponding holonomies. 
\begin{lemma}\label{lem:curvature:global}
For an arbitrary $A$-connection $\nablaE$ on a vector bundle $E\to M$ with curvature $\omegaE\in\Omega^2(A,\Hom(E,E))$, and any $A$-homotopy $adt+bds:TI^2\to A$, the       
parallel transport and the curvature are related as follows:                                                                                                                  
\begin{equation}\label{eq:curvature:global}                                                                                                                                   
\frac{d}{ds} \holE_{a_{1,0}^s}(x)=\int_0^1 \holE_{a^s_{1,t}}\circ\ \omegaE(a,b)_{\gamma_{t}^{s}} \circ \holE_{a^s_{t,0}}(x)dt,                                            
\end{equation}                                                                                                                                                                
\end{lemma}
A detailed proof of Lemma \ref{lem:curvature:global} can be found in \cite{NI01} in the case of usual linear connections (\emph{i.e.} $TM$-connection). The case of an arbitrary $A$-connection follows by pulling back the $A$-connection along morphisms $TI^2\to A$.

\begin{remark}
Notice the global nature of the equation \eqref{eq:curvature:global} as opposed to the local definition of the curvature $\omegaE=[\nablaE,\nablaE]-\nablaE_{[\ , \ ]}$. It seems not to be a very popular equation\footnote{see the comment in \cite[Sec. 15.4.1]{BE01}} although it goes back to the work of Nijenhuis \cite{NI01} and  a fundamental one for our purposes.
\end{remark}
\begin{lemma}
	 The holonomy along an $A$-path only depends on its thin homotopy class.  
\end{lemma}
\begin{proof}
	Given a thin homotopy $adt+bds:TI^2\to A$, the term $\omegaE(\alpha,\beta)$ in the right hand term of the equation  \eqref{eq:curvature:global} vanishes. It follows that $\holE_{a^s_{1,0}}$ is independent of $s$, and similarly for $\holC_{a^s_{1,0}}$. We conclude that $\hol(a)=(\holE_a,\holC_a)$ only depends on the thin homotopy class of $a$.
\end{proof}

In order to keep simple notations, in the sequel we denote $a_0$ and $a_1$ the paths $a^{s=0}_{1,0}$ and $a^{s=1}_{1,0}$ (notice that the lower indice then refers to the $s$ variable). This way, $\sig=adt+bds$ is a $A$-homotopy $\sigma:a_0\overset{}{\Rightarrow}a_1$.

\begin{lemma}\label{lem:hom-fonctorial1}
Given an $A$-homotopy $\sig:a_0\overset{}{\Rightarrow}a_1$, $\hol(\sig)$ defines a chain homotopy $\hol(a_0)\Rightarrow \hol(a_1)$  in $2\text{-}\Gau(C\xrightarrow{\d}E)$.
\end{lemma}
\begin{proof}
We have to prove that:
\begin{align*}
\hol(\sig)\circ\partial&=\holC_{a_1}-\holC_{a_0}, \\
\partial\circ \hol(\sig)&=\holE_{a_1}-\holE_{a_0}.
\end{align*}
By using successively the lemma \ref{lem:chain:hom}, the equation \eqref{eq:omegas:relation} and then the lemma \ref{lem:curvature:global}, we obtain:
\begin{align*}
\int_0^1 \holC_{a^s_{1,t}}\circ\ \omega(a,b)_{\gamma_{t}^{s}} \circ \holE_{a^s_{t,0}}dt \circ \d 
        &=\frac{d}{ds} \holC_{a^s_{1,0}},\\
\d\circ\int_0^1 \holC_{a^s_{1,t}}\circ\ \omega(a,b)_{\gamma_{t}^{s}} \circ \holE_{a^s_{t,0}}dt
        &=\frac{d}{ds} \holE_{a^s_{1,0}}.
\end{align*}
The result then follows by integrating along the $s$ variable. Note that this lemma  was proved for $A=TM$ in \cite[Prop.\, 3.13]{AriasCrainic}, where it is stated in a slightly different way.
\end{proof}

\begin{lemma}
The holonomy commutes with the vertical inverse and composition of $2$-morphisms.          
\end{lemma}
\begin{proof}
Consider two $A$-homotopies $a_0 \stackrel{\sig}{\Rightarrow} a_1\stackrel{\sig'}{\Rightarrow} a_2$. Recall that the vertical composition is given by concatenation along the $s$ variable.
It follows from the additivity of the integral that:
\begin{align*} \hol\bigl(\sig \circV \sig'\bigr)=&\int_0^1 \!\!\!  \int_0^1 \hol_{a_{1,t}^s}^C\circ\  \omega(a,b)_{\gamma_{t}^{s}}\circ \holE_{a_{t,0}^s} dt ds\\
                         &+\int_0^1 \!\!\! \int_0^1 \hol_{{a'}_{1,t}^s}^C\circ\  \omega(a',b')_{{\gamma'}_{t}^{s}}\circ \holE_{{a'}_{t,0}^s} dt ds
                         =\hol(\sig)\circV\hol(\sig').
\end{align*}
The fact that the holonomy commutes with the vertical inverse is obvious.
\end{proof}

\begin{lemma}
The holonomy commutes with the horizontal inverse and composition of $2$-morphisms.
\end{lemma}
\begin{proof}
Given $a_0 \stackrel{\sig}{\Rightarrow} a_1$ and $a'_0\stackrel{\sig'}{\Rightarrow} a'_1$ (where $a_0$ and $a'_0$ (resp. $a_1$ and $a'_1$) are composable paths) the horizontal composition is obtained by concatenation along the $t$ coordinate.

Using the additivity of the integral, we see that:
\begin{multline*}
\hol\bigl(h\circH h'\bigr)=\underbrace{\int_0^1\!\!\! \int_0^1 \holC_{a'^s_{1,0}}\circ\hol_{a_{1,t}^s}^C\circ\  \omega(a,b)_{\gamma_{t}^{s}}\circ \holE_{a_{t,0}^s} dt ds}_{\displaystyle{=:A}}\\
                 +\underbrace{\int_0^1\!\!\! \int_0^1 \holC_{{a'}_{1,t'}^{s'}}\circ\  \omega(a',b')_{\gamma_{t'}^{s'}}\circ \holE_{{a'}_{t',0}^{s'}}\circ \holE_{a_{1,0}^{s'}} dt' ds'.}_{\displaystyle{=:B}}
\end{multline*}
In the term $A$, we substitute the following relation:
\begin{align*}
\holC_{a'^s_{1,0}}
&=\holC_{{a'}^{s=0}_{1,0}}+\int_0^s\!\!\!\int_0^1 \holC_{{a'}^{s'}_{1,t'}}\circ\ \omegaC(a',b')\circ\holC_{{a'}_{t',0}^{s'}}dt'ds'\\
&=\holC_{{a'}^{s=0}_{1,0}}+\int_0^s\!\!\!\int_0^1 \holC_{{a'}^{s'}_{1,t'}}\circ\ \omega(a',b')\circ\holE_{{a'}_{t',0}^{s'}}dt'ds'\circ \d.
\end{align*}
We obtain:
\begin{align*}
A=&\int_0^1\!\!\!\int_0^1 \holC_{{a'}^{s=0}_{1,0}} \circ \holC_{a_{1,t}^s}\circ\  \omega(a,b)_{\gamma_{t}^{s}}\circ \holE_{{a}_{t,0}^s} dt ds\\
  &\quad+\int_0^1\!\!\!\int_0^1 \Bigl(\, \int_0^s\!\!\!\int_0^1 \holC_{{a'}^{s'}_{1,t'}}\circ\ \omega(a',b')\circ\holE_{{a'}_{t',0}^{s'}}dt'ds'\Bigr)\circ \d \circ\holC_{a_{1,t}^s}\circ\  \omega(a,b)_{\gamma_{t}^{s}}\circ \holE_{{a}_{t,0}^s} dt ds\\
¨
 =& \holC_{{a'}^{s=0}_{1,0}}\circ\int_0^1\!\!\! \int_0^1 \holC_{a_{1,t}^s}\circ\  \omega(a,b)_{\gamma_{t}^{s}}\circ \holE_{{a}_{t,0}^s} dt ds\\
  &\quad+\int_0^1\!\!\! \int_0^1 \holC_{{a'}^{s'}_{1,t'}}\circ\ \omega(a',b')_{\gamma_{t'}^{s'}}\circ\holE_{{a'}_{t',0}^{s'}}\circ\ \Bigl(\,\int_{s'}^1\!\!\!\int_0^1\holE_{a_{1,t}^s}\circ\ \omegaE(a,b)_{\gamma_{t}^{s}}\circ \holE_{{a}_{t,0}^s}dtds\Bigr) dt' ds'\\
 =&\holC_{{a'}^{s=0}_{1,0}}\circ\int_0^1\!\!\! \int_0^1 \holC_{a_{1,t}^s}\circ\  \omega(a,b)_{\gamma_{t}^{s}}\circ \holE_{{a}_{t,0}^s} dt ds\\
  &\quad+\Bigl(\,\int_0^1\!\!\! \int_0^1 \holC_{{a'}^{s'}_{1,t'}}\circ\ \omega(a',b')_{\gamma_{t'}^{s'}}\circ\holE_{{a'}_{t',0}^{s'}} dt' ds'\Bigl)\circ \holE_{a^{s=1}_{1,0}}\\
  &\quad\quad-\int_0^1\!\!\! \int_0^1 \holC_{{a'}^{s'}_{1,t'}}\circ\ \omega(a',b')_{\gamma_{t'}^{s'}}\circ\holE_{{a'}_{t',0}^{s'}}\circ\holE_{a^{s=s'}_{1,0}} dt' ds'\\
 =&\holC(a'_{0})\circ\hol(\sig)+\hol(\sig')\circ\hol(a_{1})- B
\end{align*}
We conclude that: $\hol(\sig' \circH \sig)=\holC(a'_{0})\circ\hol(\sig)+\hol(\sig')\circ\holE(a_{1})=\hol(\sig' )\circH \hol(\sig)$. We leave it to the reader to check that $\hol$ commutes with the horizontal inverses.
\end{proof}

\begin{lemma}\label{lem:3homotopyindependant}
$\hol(\sig)$ is independent of the $3$-homotopy class of $\sig$.      
\end{lemma}
\begin{proof}
Given a homotopy $H=adt+bds+cdu:TI^3\to A$, we denote $\gamma^{s,u}_t$ rather than $\gamma(t,s,u)$ the base path. Then $h_0$ and $h_u$ will denote respectively the $A$-homotopie $h_0:=(adt+bds)|_{I^2\times\{u\}}$ so that $H$ is a homotopy $H:h_0\Rightarrow h_1$. Also $a_{t',t}^{s,u}$ will refer to the path $a|_{[t,t']\times\{s\}\times\{u\}}dt$, while:
\begin{align*}\holE_{a^{s,u}_{t',t}}&:E_{\gamma^{s,u}_t}\to E_{\gamma^{s,u}_{t'}}, \\
                \holC_{a^{s,u}_{t,0}}&:C_{\gamma^{s,u}_t}\to C_{\gamma^{s,u}_{t'}}
\end{align*}
will denote the corresponding holonomies. We will show that $\hol(h_u)$ is independant of $u$, which is  a consequence of \eqref{eq:ruh:omegaclosed}. For this, we compute:
\begin{align*}
\frac{d}{du} \hol(h_u)
&=\int_0^1\!\!\!\int_0^1 \frac{d}{du}\holC_{a^{s,u}_{1,t}}\circ\ \omega(a,b)_{\gamma^{s,u}_t}\circ \holE_{a^{s,u}_{t,0}}dtds\\
&=\int_0^1\!\!\!\int_0^1 \holC_{a^{s,u}_{1,t}}\circ\,\nabla_u\, \omega(a,b)_{\gamma^{s,u}_t}\circ \holE_{a^{s,u}_{t,0}}dtds\\
&=-\underbrace{\int_0^1\!\!\!\int_0^1 \holC_{a^{s,u}_{1,t}}\circ\,\nabla_a\, \omega(b,c)_{\gamma^{s,u}_t}\circ \holE_{a^{s,u}_{t,0}}dtds}_{=:A}\\
&\quad\quad\quad\quad\quad\quad-\underbrace{\int_0^1\!\!\!\int_0^1 \holC_{a^{s,u}_{1,t}}\circ\,\nabla_b\,\omega(c,a)_{\gamma^{s,u}_t}\circ \holE_{a^{s,u}_{t,0}}dtds.}_{=:B}
\end{align*} 
On the right hand side, we have:
\begin{align*}
B&=\int_0^1\!\!\!\int_0^1 \holC_{a^{s,u}_{1,t}}\circ\,\nabla_b\, \omega(c,a)_{\gamma^{s,u}_t}\circ \holE_{a^{s,u}_{t,0}}dtds\\
 &=\int_0^1\!\!\!\int_0^1 \frac{d}{ds}\holC_{a^{s,u}_{1,t}}\circ\, \omega(c,a)_{\gamma^{s,u}_t}\circ \holE_{a^{s,u}_{t,0}}dtds\\
 &=\int_0^1\Bigl(\holC_{a^{1,u}_{1,t}}\circ\, \omega(c,a)_{\gamma^{1,u}_t}\circ \holE_{a^{1,u}_{t,0}}-\holC_{a^{0,u}_{1,t}}\circ\,\omega(c,a)_{\gamma^{0,u}_t}\circ \holE_{a^{0,u}_{t,0}}\Bigr)dt=0.
\end{align*}Also:
\begin{align*}
A&=\int_0^1\!\!\!\int_0^1 \holC_{a^{s,u}_{1,t}}\circ \holC_{b_{s,1}^{u,t}} \circ\,\Bigl( \holC_{b_{1,s}^{u,t}}\circ\,\nabla_a\,\omega(b,c)_{\gamma^{s,u}_t} \circ\holE_{b_{s,0}^{u,t}}\Bigr)\circ\holE_{b_{0,s}^{u,t}}\circ\holE_{a^{s,u}_{t,0}}dtds\\
 &=\int_0^1\!\!\!\int_0^1 \holC_{a^{s,u}_{1,t}}\circ \holC_{b_{s,1}^{u,t}} \circ\,\Bigl(\frac{d}{dt} \holC_{b_{1,s}^{u,t}}\circ\,\omega(b,c)_{\gamma^{s,u}_t} \circ\holE_{b_{s,0}^{u,t}}\Bigr)\circ\holE_{b_{0,s}^{u,t}}\circ \holE_{a^{s,u}_{t,0}}dtds\\
 &=\int_0^1\!\!\!\int_0^1 \frac{d}{dt} \holC_{a^{s,u}_{1,t}}\circ\ \omega(b,c)_{\gamma^{s,u}_t} \circ \holE_{a^{s,u}_{t,0}}dtds\\
 &=\int_0^1\Bigl(\omega(b,c)_{\gamma^{s,u}_{1}} \circ \holE_{a^{s,u}_{1,0}}-\holC_{a^{s,u}_{1,0}}\circ\, \omega(b,c)_{\gamma^{s,u}_{0}}  \Bigl)ds=0,
\end{align*}
which completes the proof.
\end{proof}

\section{A geometric description of the semi-direct product}{\label{sec:geometric-description}}
It is possible to illustrate geometrically the construction of the transformation 2-groupoid $\mathcal{P}_2(A)\ltimes \E$ by a series of explicit diagrams. Although essentially informal, these diagrams usually offer guidance in the understanding of the algebraic construction.

In order to explain this, we shall use systematically the following notations: given a $1$-morphism $(c,a,e)\in \mathcal{P}_1(A)\ltimes \E$, we denote by $e'$ and $\aK$ the following elements: $e':=\holE_a(e)\in E$ and  $\aK:=(c,e')\in \G(\K)=E\oplus C$ (and similarly when indices are involved).

The geometric illustration goes as follows: in a $1$-morphism $(c,a,e)$, the couple $(a,e) $ should be thought of as the horizontal lift $\tilde{a}$ of $a$ with source $e$. Hence $(a,e)$ has target $e'$. Then $(c,a,e)$ behaves like the composition of $\tilde{a}$ with $\aK$. In other words, one may think of $(c,a,e)$ as a 'formal concatenation'  $(c,a,e)= v\cdot \tilde{a}$. This can be pictured as follows:
\diagramh{\begin{equation}\label{diag:splipath}
\begin{gathered}
                   \SelectTips{cm}{}\xymatrix@R=10pt{    &  &                                        &    *+[r]{  e'+\d c }\\
                   E\ar[dd]_{\pE}   &   &                                                                        \\
                                   &    &   e   \ar@/^0.5pc/[r]|{\, \tilde{a}\, } \ar@/^0,5pc/@{-->}[ruu]^{(c,a,e)}  &   *+[r]{e':=\holE_a(e)} \ar@/^0.5pc/[uu]|{\quad\ \ \aK=(e',c)}\\
                   M               &    &             m   \ar@/^0.5pc/[r]|{\, a\, }  &      m' }
\end{gathered}
\end{equation}}
Using the above diagram as a guidance, we easily recover the source and target maps on the space of $1$-morphisms. One can also recover the composition of $1$-morphisms, by a concatenation process as follows:
\diagramh{$$\SelectTips{cm}{}\xymatrix@R=35pt{    
                                                 &                   &*+[r]{e_1'+\d c'}\\
                \quad\quad\quad\quad             & e_0'+\d c_0    \ar@/^0,5pc/[r]|{\,\tilde{a}_1\,} \ar@/^0,5pc/@{-->}[ru]^{(c_1,a_1,x_1)}& *+[r]{e_1'} \ar@/^0,5pc/[u]|{\aK_1}   \\
                e_0    \ar@/^0,5pc/[r]|{\,\tilde{a}_0\,}\ar@/^0,5pc/@{-->}[ru]^{(c_0,a_0,e_0)}
								                                  & e_0'           \ar@/^0,5pc/[u]|{\ \aK_0}\ar@/^0,5pc/@{..>}[r] 
																									                    & *+[r]{\holE_{a_1}(e'_0)} \ar@/^0,5pc/@{..>}[u]|{\holK_{a_1}(\aK_0)\ }\ar@/_2pc/@{..>}[uu]|{\quad\quad\quad\! \aK_1\cdot \holK_{a_1}(\aK_0)} \\
                m_0    \ar@/^0,5pc/[r]|{\,a_0\,} \ar@/_1,5pc/@{..>}[rr]|{\, a_1\cdot a_0\,} 
                                                  & m_0'           \ar@/^0,5pc/[r]|{\,a_1\,}& m''}$$}
Then the inversion of $1$-morphisms can be recovered in an easy manner.

In order to illustrate illustrate the space of $2$-morphisms, given a $2$-morphism $(c_0,\sig,e_0)\in P_2(A)\ltimes \E= C \ {}_{\pC}\!\!\times_{\tH} P_2(A)\ {}_{\sH}\!\!\times_{\pE} E$, we shall use the following notations:
\begin{itemize}
\item $a_0:=\sV(\sig)\in P_1(A)$ denotes the source of $\sigma$,
\item $v_0:=(e'_0,c_0)\in E\oplus C$, where $e'_0:=\holE_a(e_0)$,
\item $a_1:=\tV(\sig)\in P_1(A)$ denotes the target of $\sigma$,
\item $v_1:=(e'_1,c_1)\in E\oplus C$, where $e'_1:=\holE_a(e_1)$, and $e_1:=e_0$.
\end{itemize}
The idea lying behind $2$-morphisms is the following. First we lift $a_0, a_1$ into $D$-paths $\tilde{a}_0, \tilde{a}_1$, starting at a same point $e_0\in E$. Although $a_0$ and $a_1$ are homotopic $A$-paths,  because of the presence of curvature, their horizontal lifts $\tilde{a}_0, \tilde{a}_1$ are not homotopic in $D$. Since this is precisely what $\hol(\sig)$ measures, we obtain $(c_0,h,e_0)=:\tilde{\sig}$ as a formal homotopy between  $\aK_0\cdot a_0=(c_0, a_0, e_0)$ and $\aK_1\cdot a_1=(c_1,a_1,e_0)$ provided we set $c_1=c_0-\hol(\sig)_{e_0}$, as suggested by the following diagram:
\diagramh{
\begin{align}\label{diag:splithomm}
\begin{aligned}
&\SelectTips{cm}{}\xymatrix@R=15pt@C=15pt{ 
                                   &&    &&*+[r]{e_0'+\d c_0 }&                   \\
																	                                                    \\
                                   &&   & &      \\
																	 &&   & &e'_1\ar@/_0,5pc/[uuu]|{\ \aK_1}&                                                                       \\
                                   &e_0=e_1  \ar@/_0,5pc/[rrrr]|>>>>>>>>>{\,\tilde{a}_0\,}
																	           \ar@/_0,5pc/[rrru]|{\,\tilde{a}_1\,}
																	           \ar@/_1,5pc/@{-->}[rrruuuu]|{\quad\overset{\ }{\underset{\ }{(c_0,a_0,e_0)}}}="f"
																						 \ar@/^1,5pc/@{-->}[rrruuuu]|{\underset{\ }{(c_1,a_1,e_1)}\quad}="g"
																						 &&&& e_0'  \ar@/_0,5pc/[uuuul]|{\ \aK_0}
																						            \ar@/_0,5pc/@{..>}[ul]^{\hol(\sig)_{e_0}}       \ar@{=>}^-{\tilde{\sig}}"f";"g"-<-1pt,4pt>}\\
&\xymatrix{                    \\  & \ \  m_0   \ar@/_1,5pc/[rrr]|>>>>>>{\,\overset{\ }{a_0}\,}="d" 
                                                \ar@/^1,5pc/[rrr]|<<<<<<<<{\,a_1\,}="e"
																							&\ &\quad&     {m_0'}\quad\quad                    \ar@{=>}^-{\sig}"d";"e"-<-1pt,4pt>
																							                                    }
\end{aligned}\end{align}}

In this way we easily recover the source and target maps on the space of $2$-morphisms. Then the vertical composition of $2$-morphisms can be depicted by a concatenation process:
\diagramh{
\begin{align*}
\SelectTips{cm}{}\xymatrix@R=25pt@C=15pt{ 
                                              &          & & &     &{e_0'+\d c_0 }&                                      \\
																	            &          & & &                                                           \\
                                              &          & & &                                                           \\
																	            &          & & &e'_2   \ar@/_0,5pc/[uuur]|{\aK_2}&                      \\
                                   e_0=e_1=e_2 
																	              \ar@/_1pc/[rrrrrrd]|<<<<<<<<<<<<<<<<<<<<<{\,\tilde{a}_0\,}
																	              \ar@/_0,5pc/[rrrrr]|{\,\tilde{a}_1\,}
																	              \ar@/_0.5pc/[rrrru]|<<<<<<<<{\,\tilde{a}_2\,}
																								\ar@/^2,5pc/@{-->}[rrrrruuuu]|{\underset{\ }{(c_2,a_2,e_2)}}="u"
																								\ar@{-->}[rrrrruuuu]|{\underset{\ }{(c_1,a_1,e_1)}}="v"
																								\ar@/_2,5pc/@{-->}[rrrrruuuu]|{\overset{\ }{(c_0,a_0,e_0)}}="w"        \ar@{=>}_-{}"w";"v"-<5pt,4pt>
																								                                                                 \ar@{=>}_-{}"v";"u"-<5pt,4pt>
																					    &\quad\quad\quad& & &       & e_1' 
																							                  \ar@/_0,5pc/[uuuu]|{\ \aK_1}
																																\ar@/_0,5pc/@{..>}[ul]^{\hol(\sig)_{e_1}}         \\
															                & & & &     &               & e'_0   \ar@/_1pc/[uuuuul]|{\ \aK_0}\ar@/_0,5pc/@{..>}[ul]^{\hol(\sig)_{e_0}}     \\
																   m_0          \ar@/_2pc/[rrrrr]|>>>>>>>>>>>>>{\,\overset{\ }{a_0}\,}="c"
	                                              \ar[rrrrr]|{\,\overset{\ }{\underset{\ }{\,a_1\,}}}="d"
										                            \ar@/^2pc/[rrrrr]|<<<<<<<<<<<<<{\,a_2\,}="e"
	                                           &                & & &       & {m_0'}\ar@{=>}_-{\sig_2}"d";"e"-<-5pt,4pt>
						                                                  \ar@{=>}^-{\sig_1}"c";"d"-<-5pt,4pt>}
\end{align*}}
We leave it to the reader to figure out the horizontal composition.

\end{document}